\documentclass[a4paper,11pt]{amsart}

\usepackage{amsmath}
\usepackage{amssymb}
\usepackage[english]{babel}
\usepackage[applemac]{inputenc}
\usepackage{MnSymbol}
\usepackage{hyperref}
\usepackage{xcolor}

\numberwithin{equation}{section}

\renewcommand{\geq}{\geqslant}
\renewcommand{\leq}{\leqslant}
\newcommand{\R}{{\mathbb R}}
\newcommand{\C}{{\mathbb C}}
\newcommand{\Z}{{\mathbb Z}}

\newcommand{\T}{{\mathbb T}}

\newcommand{\PP}{{\mathbb P}}

\newcommand{\tor}{{\mathfrak t}}

\newcommand{\Pol}{\Sigma}
\newcommand{\tr}{\mathrm{tr}}

\newcommand{\gr}{\mathrm{grad}}

\newcommand{\cL}{{\mathcal L}}% Lie derivative
\newcommand{\cO}{{\mathcal O}}% coherent sheaves
\newcommand{\cK}{{\mathcal K}}%canonical bundle

\newcommand{\Scal}{\mathit{Scal}}
\newcommand{\Sph}{{\mathbb S}}% sphere
\newcommand{\La}{{\mathrm L}}
\def\th/#1#2{{#1}^{#2}}% transversal holomorphic structure

\newtheorem{prop}{Proposition}

\newtheorem{lemma}[prop]{Lemma}

\newtheorem{thm}[prop]{Theorem}
\newtheorem{cor}[prop]{Corollary}
\newtheorem{conj}[prop]{Conjecture}
\numberwithin{prop}{section}
\newtheorem{conclusion}{Conclusion}

\theoremstyle{definition}
\newtheorem{dfn}[prop]{Definition}
\newtheorem{ex}[prop]{Example}
\newtheorem{rem}[prop]{Remark}

\def\cal{\mathcal }

\def\fnote#1{\footnote}

\def\be{\begin{equation}}
\def\ee{\end{equation}}
\def\bea{\begin{eqnarray*}}
\def\eea{\end{eqnarray*}}

\date{\today}

\begin{document}

\title[Steady GKRS's]{Hamiltonian $2$-forms and  new explicit Calabi--Yau metrics and gradient steady K\"ahler--Ricci solitons on $\C^n$}

\author[V. Apostolov]{Vestislav Apostolov} 
\address{Vestislav Apostolov \\ Laboratoire Jean Leray, Nantes Universit\'e \\ and \\
Institute of Mathematics and Informatics\\  Bulgarian Academy of Sciences \\ and \\
Department of Mathematics,  UQAM \\ CP 8888, succ. Centre-ville \\
Montreal (QC) H3C 3P8, Canada \\} 
\email{vestislav.apostolov@uqam.ca}

\author[C. Cifarelli]{Charles Cifarelli} \address{Charles Cifarelli \\ 
Laboratoire Jean Leray, Nantes Universit\'e\\
2, Rue de la Houssini\`ere - BP 92208\\
F-44322 Nantes, France \\} 
\email{Charles.Cifarelli@univ-nantes.fr}

\thanks{V.A.  was supported in part by an NSERC Discovery Grant. Both authors were supported  by a ``Connect Talent'' Grant of the R\'egion des Pays de la Loire in France. }

\date{\today}

\begin{abstract} For each partition of the positive integer $n= \ell +\sum_{j=1}^\ell d_j$,  where  $\ell\ge 1$ and $d_j \ge 0$  are integers, we construct a continuous $(\ell-1)$-parameter family of  explicit complete gradient steady K\"ahler--Ricci solitons on $\C^n$ admitting a hamiltonian $2$-form of order $\ell$ and symmetry  group ${\rm U}(d_1+ 1) \times \cdots \times {\rm U}(d_{\ell} + 1)$.  For $\ell=1, \, d_1=n-1$ we obtain Cao's example \cite{Cao} whereas for other partitions the metrics are new. Furthermore, when $n=2, \, \ell=2, \, d_1=d_2=0$ we obtain complete gradient steady K\"ahler--Ricci solitons on $\C^2$ which have positive sectional curvature but are not isometric to Cao's ${\rm U}(2)$-invariant example.  This disproves a conjecture by Cao. We also present a construction yielding explicit families of complete gradient steady K\"ahler-Ricci solitons on $\C^n$ containing higher dimensional extensions of the Taub-NUT Ricci-flat K\"ahler metric on $\C^2$. When $n\ge 3$, the complete Ricci-flat K\"ahler metrics,  and  when $n\ge 2$,  their deformations to complete gradient steady K\"ahler-Ricci solitons seem not to have been observed before our work.
\end{abstract}

\maketitle

\section{Introduction} A \emph{steady  K\"ahler-Ricci soliton} is a K\"ahler manifold $(M, J, g, \omega)$ endowed with a (real) holomorphic vector field $X$ such that  the Ricci form $\rho$ of $\omega$ satisfies 
\begin{equation}\label{steadyKRS}
 \rho = -\frac{1}{2} {\mathcal L}_{JX} \omega, \end{equation}
where ${\mathcal L}$ denotes the Lie derivative. In the case when $X=0$, we obtain a \emph{Ricci-flat} K\"ahler metric.

We will be interested  throughout this paper in the special case when $X= -\omega^{-1}(df)$. Then $X$ is also symplectic and hence a Killing vector field of the corresponding K\"ahler structure $(g, J)$. In this case,  the  Ricci tensor $Ric_g$ satisfies
\[   Ric_g = \frac{1}{2}{\mathcal L}_{\nabla f} g.\]
The data $(M, J, \omega, X)$ is  referred to  as a \emph{steady gradient K\"ahler-Ricci soliton} (steady GKRS for short). 

In the case when $g$ is a \emph{complete} riemannian metric, the vector field $JX$ is complete too \cite{zhang-complete}  and acting on $g$ by the flow of $\frac{1}{2}JX$ gives rise to a self-similar solution of the much studied K\"ahler-Ricci flow %\footnote{CC: Changed a factor of 1/2 to make consistent with the usual normalization for KRF} 
\[ \frac{\partial g_t}{\partial t} = - Ric_{g_t}. \]
Thus, complete steady GKRS's represent models for the possible formation of singularities of the K\"ahler-Ricci flow and have been intensively studied in recent times.  Notice that  a complete steady GKRS is necessarily \emph{non-compact} (see, for example, \cite[Proposition 2.1]{XZhang-compact}).

A comprehensive study of  steady GKRS's  from a local prospective was carried out by R. Bryant~\cite{bryant}, who showed that to each complex $n$-dimensional simply connected steady GKRS $(M, J, \omega, X)$ one can associate a unique up to a multiplicative constant holomorphic volume form $\Theta \in \Omega^{n, 0}(M, J)$,  such that \[ dV_g =\frac{\omega^{n}}{n!} = (\sqrt{-1})^{n^2} C  e^{-f}\Theta \wedge \bar \Theta, \qquad C>0. \] Together with the holomorphic vector field $Z= X - iJX$, these are the only  \emph{local} invariants of a steady GKRS around a point where $Z\neq 0$: in a small neighbourhood of such a point, there exists a germ of steady GKRS with prescribed invariants $(Z, \Theta)$, depending on the choice of an arbitrary real analytic K\"ahler metric and a function defined on a complex hyper-surface  of $M$,   transversal to $Z$.

On the other hand, \emph{complete} examples of steady GKRS's are still scarce.  Using the Calabi ansatz, Cao~\cite{Cao} constructed a complete ${\rm U}(n)$-invariant steady GKRS on $\C^n$, and showed that it is the unique (up to isometry) steady GKRS on $\C^n$ with this symmetry group.  Similar constructions produced complete examples on the total spaces of (roots of) the canonical bundle over a Fano K\"ahler-Einstein base \cite{Cao,FW}, and certain generalizations \cite{Schafer1}.  Biquard--Macbeth~\cite{Biquard-Macbeth} have used gluing techniques and   Conlon--Deruelle~\cite{ConDer}  and Sch\"afer~\cite{Schafer2} have used PDE methods to obtain further complete examples.

Regarding possible classification results, Cao observed  in \cite{Cao} that the  ${\rm U}(n)$-invariant steady GKRS on $\C^n$ has everywhere positive sectional curvature and he conjectured
\begin{conj}\label{c:Cao}\cite{Cao} Any simply connected complete steady gradient K\"ahler-Ricci soliton with positive sectional curvature is isometric to the ${\rm U}(n)$-invariant steady GKRS on $\C^n$.
\end{conj}
Partial evidence for this conjecture was provided in a consequent work by Cao--Hamilton~\cite{Cao-Hamilton},   who showed that a complete steady GKRS with positive Ricci curvature  and scalar curvature having a critical point must be diffeomorphic to $\R^{2n}$.
Bryant~\cite{bryant} and Chau--Tam~\cite{Chau-Tam} independently improved this result by showing that such a GKRS must be, furthermore,  biholomorphic to $\C^n$.

\bigskip
Thus motivated, we use the theory of \emph{Hamiltonian 2-forms} developed in  \cite{ACG,ACGT} as an ansatz in order to construct  new families of  \emph{explicit}  complete  steady GKRS's on $\C^n$.  On each $\C^n$, our construction actually gives families of two different geometric flavours. 
 
 \smallskip
 The first class of examples, which we call of \emph{Cao type}, arises from any partition of $n$  as sum of non-negative integers
 \[n= \ell +\sum_{j=1}^\ell d_j, \qquad \ell \ge 1, \qquad d_j \ge 0. \]
 \begin{thm}[Cao's type GKRS]\label{thm:Cao} Given a partition of the integer $n\ge 2$ as above, there exists an $(\ell-1)$-dimensional family of non-isometric, complete  irreducible steady GKRS on $\C^n$, all admitting a hamiltonian $2$-form of order $\ell$ and isometry group $\prod_{j=1}^{\ell} {\rm U}(d_j+1)$.  
 \end{thm}
By irreducible, we mean that the K\"ahler metric is not a product of K\"ahler structures  of lower dimension. We notice that the different families  of steady GKRS obtained from the above result do not contain isometric metrics.  In the case when $\ell=1, d_1=n-1$, we obtain Cao's GKRS on $\C^n$ and we show in Lemma~\ref{l:volume} that, in general, the volume of the riemannian balls of radius $r$ grows when $r\to \infty$ as $r^n$.  Furthermore,  for $n=2, \ell=2,  d_1=d_2=0$, we show that  the corresponding complete GKRS's on $\C^2$ have positive sectional curvature. 
 \begin{cor}\label{c:Cao-conjecture} $\C^2$ admits a $1$-parameter family of complete irreducible steady GKRS's with ${\rm U}(1) \times {\rm U}(1)$-symmetry, which have positive sectional curvature and whose scalar curvature achieves its maximal value at the origin of $\C^2$. In particular, there are infinitely many steady GKRS's of positive sectional curvature and volume growth $r^2$,  which are not isometric to Cao's  ${\rm U}(2)$-invariant example.
 \end{cor}
 
 \smallskip
 Our second series of examples, which we call of \emph{Taub-NUT type},  arises from a partition
 \[ n = \ell + \sum_{j=1}^{\ell-1} d_j, \qquad  \ell \ge 2, \quad d_j \ge 0.\]
 In this case we have
 \begin{thm}[Taub-NUT type GKRS]\label{thm:Taub-NUT} Given a partition of the integer $n\ge 2$ as above, there exists an $(\ell-1)$-dimensional family of non-isometric, irreducible complete steady GKRS on $\C^n$, all admitting a hamiltonian $2$-form of order $\ell$ and isometry group ${\rm U}(1)\times \prod_{j=1}^{\ell-1} {\rm U}(d_j+1)$.  An $(\ell-2)$-dimensional sub-family consists of complete Ricci-flat K\"ahler metrics on $\C^n$.  \end{thm}
We show that the above complete metrics have a volume growth of riemannian balls of radius $r$ (as $r\to \infty$) of  the rate  of $r^{2n-1}$ (see Lemma~\ref{l:volume}).  When $n\ge 3$,  the explicit complete Calabi-Yau metrics on $\C^n$ seem not to have been observed in the mathematical literature. Indeed, they have a  different volume growth  than the Taub-NUT type Calabi-Yau K\"ahler metrics recently constructed on $\C^3$  by Y. Li  \cite{Li} (see in particular Corollary 2.27 in \cite{Li}), as well as the Calabi-Yau examples with euclidean volume growth on $\C^n$ found in \cite{CR,Li0,Sz}.   In the case  $n=2, \,  \ell=2, \, d_1=0$,   the Ricci-flat metric corresponds to the Taub-NUT metric on $\C^2$~\cite{LeBrun}. We thus get
 \begin{cor}\label{c:Taub-NUT} $\C^2$ admits a one-parameter family of complete ${\rm U}(1)\times {\rm U}(1)$-invariant steady GKRS's, including the Taub-NUT Ricci-flat K\"ahler metric.  \end{cor}

\smallskip
The paper is organized as follows. In Section~\ref{s:toric},  we review the general theory~\cite{Abreu1, guillemin} of invariant K\"ahler metrics on toric manifolds in adapted momentum-angular coordinates. In particular, we recast the steady GKRS equation and provide necessary and sufficient conditions for the corresponding complex structure to be standard, i.e. arising from the fan. Section~\ref{s:calabi-type} recasts the well-studied steady GKRS's obtained from the Calabi ansatz. Here we introduce Cao's example on $\C^n$ from the point of view of the profile function~\cite{HS}, and present its toric description. Section~\ref{s:hfkg} introduces the  general form of K\"ahler metrics with a hamiltonian $2$-form of order $\ell\ge 1$ (obtained in \cite{ACGT1}) which we use in this paper as an ansatz for the construction of explicit GKRS's (Lemma~\ref{l:KRS-ell}). This construction includes the Calabi ansatz as the special case $\ell=1$.  We then study the hamiltonian $2$-form ansatz with suitable initial data and, using the general toric formalism developed in Section~\ref{s:toric}, we show that it defines the desired global metrics on $\C^n$, completing the proofs of Theorems~\ref{thm:Cao} and \ref{thm:Taub-NUT}. Finally, in Section~\ref{s:D2} we specialize our construction to $\C^2$ and $\ell=2$, in which case the new families of GKRS become  \emph{orthotoric}~\cite{ACG0} and are thus a special case of the \emph{ambitoric} structures classified in \cite{ACG2}. This observation allows us to spell out explicitly various geometric quantities, such as the soliton vector field, and the global K\"ahler potential of the metrics. This also allows us to perform exhaustive curvature computations,  presented in the Appendix, and thus to complete the proofs of  the Corollaries~\ref{c:Cao-conjecture} and \ref{c:Taub-NUT}.

\section{Toric K\"ahler metrics from the Delzant construction}\label{s:toric} 
In this section we collect some known facts about $\T^n$-invariant K\"ahler metrics defined on  a not necessarily compact symplectic $2n$-manifold $(M, \omega_0)$.  

We start with a vector space $\tor$ of real dimension $n$, endowed with a lattice $\Lambda \subset \tor$ of maximal rank $n$, and denote by $\T^n= \tor/ 2\pi \Lambda$ the corresponding compact torus. We consider a (possibly unbounded) simple closed convex polyhedron $\Pol \subset \tor^*$  of finite type which satisfies the Delzant condition. This means that  
\begin{dfn}\label{d:delzant} Suppose $\Pol$ is a polyhedron in $\tor^*$. It is called a \emph{Delzant polyhedron} with respect to the lattice $\Lambda \subset \tor$ if
\begin{enumerate}
\item[$\bullet$] $\Pol$ is generated by a \emph{finite} number of affine-linear functions \newline $L_1(x), \ldots , L_d(x)$ (called \emph{labels}), i.e. 
\[\Pol =\{ x \in \tor^* \, | \, L_j(x) \ge 0, j=1, \ldots, d\}.\]
\item[$\bullet$] $\Pol$ contains a finite non-zero number of points (vertices) such that  at each vertex $v$, there are precisely $n$ labels $L_1, \ldots, L_n$ vanishing at $v$  and, furthermore,   $\{ dL_1, \ldots,  dL_n\}$ is a basis of $\Lambda$.
\end{enumerate}
\end{dfn}
Under the above assumptions, the Delzant construction~\cite{Delzant}  associates to $(\Pol, \Lambda)$ a smooth symplectic  $2n$-manifold $(M, \omega_0)$  such that $\T^n$  acts in hamiltonian way with momentum map 
\[\mu: M \to  \tor^*, \qquad \mu(M)=\Pol.\]
\begin{ex}\label{e:Rn}
The basic example is to take
 \[\Pol_n:=\{ x\in \R^n \, | \, x_j\ge 0, \, j=1, \ldots, n\}.\]
  In this case $\Lambda=\Z^n$ and the Delzant construction gives $\R^{2n}$ endowed with its standard symplectic form $\omega_0$.
\end{ex}
The so-called \emph{Abreu--Guillemin theory}~\cite{Abreu1,guillemin} describes, in the compact case,  the set of $\T^n$-invariant K\"ahler metrics on $M$ which are compatible with the fixed symplectic form $\omega_0$ in terms of convex smooth functions defined on the interior $\mathring{\Pol}$ of $\Pol$. The works \cite{Abreu-Sena-Dias, cifarelli-uniqueness, Sena-Dias}  extend this theory to the non-compact case.  

The starting point is that the Delzant construction endows $(M, \omega_0)$ with a $\T^n$-invariant  $\omega_0$-compatible K\"ahler metric $g_0$ and complex structure $J_0$, such that in suitable momentum-angular $(d\mu_j, dt_j)$ coordinates defined on $M^0:= \mu^{-1}(\mathring{\Pol})$, one has~\cite{guillemin}
\begin{equation}\label{canonical}
\begin{split}
g_0=& \sum_{i,j=1}^n \left({\rm Hess}(u_0)\right)_{ij} d\mu_id\mu_j + \left({\rm Hess}(u_0)\right)^{-1}_{ij} dt_i dt_j, \\
\omega_0 =& \sum_{i=1}^n d\mu_i \wedge dt_i,
\end{split} \end{equation}
where $u_0(x)$ is a smooth convex function on $\mathring{\Pol}$\footnote{Throughout the paper, we fix a basis $(e_i)$ of $\tor$ (with the corresponding dual basis $(e^{i})$ of $\tor^*$),  and denote  by $x=(x_1, \ldots, x_n)$ the corresponding coordinates of $\tor$. Thus,  in the above expression, $\mu_j = \langle \mu , e_i \rangle = \mu^*(x_i)$ .}
\[ u_0(x):= \frac{1}{2} \sum_{j=1}^d L_j(x) \log L_j(x). \]
 The K\"ahler metric $(g_0, \omega_0, J_0)$ will be referred to as the \emph{canonical} K\"ahler structure.
Notice that on $(\R^{2n}, \omega_0)$ (see Example~\ref{e:Rn}), the canonical K\"ahler structure is  just the flat K\"ahler structure on $\C^n$.

Following \cite{Abreu1, guillemin}, one can use the form of \eqref{canonical} to construct more $\omega_0$-compatible $\T^n$-invariant  K\"ahler  structures on $M$ by letting $u(x)$ be a strictly convex smooth function on $\mathring{\Pol}$ and considering on $M^0$ the tensor
\begin{equation}\label{toric}
g:= \sum_{i,j=1}^n \left({\rm Hess}(u)\right)_{ij} d\mu_id\mu_j + \left({\rm Hess}(u)\right)^{-1}_{ij} dt_i dt_j.
\end{equation}
An elementary calculation shows that $g$ is $\omega_0$-compatible and the corresponding almost complex structure $J$ is integrable. However, $(g, J)$ are only defined on $M^0$. Sufficient conditions for the smooth extension to $M$ are given in \cite{Abreu-Sena-Dias}, see also the proof of  Lemma~2 and  Remark 4(ii) in \cite{ACGT1}. 
\begin{prop}\label{compactification} Suppose $u(x)$ is a strictly convex smooth function on $\mathring{\Pol}$ such that
\begin{enumerate}
\item[$\bullet$] $u-u_0$ extends to a smooth function on $\Pol$;
\item[$\bullet$] ${\rm Hess}(u) \circ {\rm Hess}(u_0)^{-1}$ extends to a smooth non-degenerate matrix-valued function on $\Pol$.
\end{enumerate}
Then $(g, J)$ is a globally defined,  smooth  $\T^n$-invariant K\"ahler structure on $M$ which is compatible with $\omega_0$.
\end{prop}
\begin{rem} Using that for the symplectic manifold $(M, \omega_0)$ obtained from a Delzant polyhedron $\Pol$ by the Delzant construction the momentum map $\mu: M \to \tor^*$ is \emph{proper},  it follows from \cite[p.~375]{Sena-Dias} that, conversely, \emph{any} globally defined $\T^n$-invariant $\omega_0$-compatible K\"ahler structure  on $M$ is equivariantly isometric to one of the form \eqref{toric} for a uniquely determined (up to the addition of an affine-linear function) convex smooth function $u$ on $\mathring{\Pol}$.
\end{rem}
\subsection{The complex structure}
As we shall see in the examples we present below, for a globally defined $\omega_0$-compatible toric  K\"ahler metric $(g, J)$ on  a \emph{non-compact} toric symplectic manifold $(M, \omega_0, \T^n)$ given by  the construction of Proposition~\ref{compactification},  the complex structures $J$ and $J_0$ are not in general biholomorphic. We formulate below a necessary and sufficient condition for the existence of a biholomorphism identifying $J$ and $J_0$, see \cite[Lemma 2.14]{cifarelli-uniqueness} and \cite[Remark 2.11]{Abreu-Sena-Dias}.
\begin{prop}\label{biholomorphic} Suppose $u$ is a strictly convex smooth function on $\mathring{\Pol}$ satisfying the conditions of Proposition~\ref{compactification} and let $J$ be the corresponding complex structure. Then $(M, J)$ is $\T^{n}$-equivariantly biholomorphic to $(M, J_0)$ if and only if for any fundamental vector field $X_v, \, v\in \tor$, the vector field $JX_v$ is complete. In this case, $(M, J)$ admits an induced holomorphic $(\C^{\times})^n$-action such that it is the toric variety classified by the fan of $\Pol$.
\end{prop}

Suppose $(g, J)$ satisfies the conditions of Proposition \ref{biholomorphic},  so that there exists an equivariant biholomorphism $\beta:(M, J_0) \to (M, J)$.  By setting $\omega = \beta^*\omega_0$, we have an identification of K\"ahler structures $(M, \omega_0, J) \cong (M, \omega, J_0)$.  Choose a base point $p \in M^0$ and a basis $(X_1, \dots, X_n)$ of $\Lambda \subset \tor$, which together determine an equivariant $J_0$-biholomorphism $M^0 \cong (\C^\times)^n$.  By a result of Guillemin \cite{guillemin},  there  exists a $\T$-invariant K\"ahler potential $H \in C^\infty( (\C^\times)^n)$ such that 
\[ \omega|_{M^0} = dd^c H. \]
By construction, our $J_0$-biholomorphsim sends the $ (\C^\times)^n$-action on $M^0$ to the standard one on $ (\C^\times)^n$.  Therefore the $\T$-invariance of $H$ implies that it can be viewed as the pullback of a smooth convex function $H(y)$ on $\R^n$ (we use exponential coordinates  $z_i= e^{y_i + \sqrt{-1}t_i}$ on the $i$-th $\C^{\times}$ factor). Notice that this description is inherently complex structure dependent, so it is crucial that we have pulled back by $\beta$.  Then a key observation going back to Guillemin \cite{guillemin} (see also \cite{abreu0}) is: 
\begin{prop}\label{legendretransform} Viewing $H(y)$ as a smooth function on $\R^n$,  up to the addition of an affine-linear function in $y$,  the gradient
 \[ x:=\nabla H =\left(\frac{\partial H}{\partial y_1}, \ldots, \frac{\partial H}{\partial y_n}\right): \R^n \to \R^n\] 
 has the property that its image is $\mathring{\Pol}$. Moreover, up to the addition of an additional affine-linear function on $\Pol$, the symplectic potential $u(x)$ is the Legendre transform of $H(y)$, i.e. $u$ and $H$ satisfy $H(y) + u(x) = \sum_{i=1}^n x_i y_i$.
\end{prop}
As an immediate corollary, we have 
\begin{cor}\label{surjectivegradient}
Suppose that $u(x)$ is a symplectic potential satisfying the conditions of Propositions \ref{compactification} and \ref{biholomorphic}. Then 
\[ y:=\nabla u =\left(\frac{\partial u}{\partial x_1}, \ldots, \frac{\partial u}{\partial x_n}\right): \mathring{\Pol} \to \R^n\]
 is a diffeomorphism.
\end{cor}
\begin{proof}
The fact that $y$ is surjective follows from Proposition~\ref{legendretransform} and the properties of the Legendre transform, as it is in fact inverse to the map $x = \left(\frac{\partial H}{\partial y_i}\right): \R^n \to \mathring{\Pol}$.  That $y$ is injective follows from the fact that $u$ is strictly convex on $\mathring{\Pol}$.  Indeed,  for any $b \in \R^n$, there cannot be two distinct points $x_0, x_1 \in \mathring{\Pol}$ which are zeros of the function $y(x) - b$, or else ${\rm Hess}(u)$ will have a degeneracy along the line $x_t = tx_1 + (1-t)x_0$.  Clearly $y: \mathring{\Pol} \to \R^n$ and its inverse $x: \R^n \to \mathring{\Pol}$ are smooth. 
\end{proof}

\begin{rem} One can  build the biholomorphism $\beta: (M, J_0) \cong (M, J)$ in practice. To this end,  for each vertex $v$ of $\Pol$, we translate $\Pol$ so that $v=0$ and assume (without loss) that the labels $L_1(x), \ldots, L_n(x)$ vanishing at $v=0$ are standard, i.e.  $L_j(x) = x_j$. With this normalization, we let 
\[y_j := \frac{\partial u}{\partial x_j}\]
 and observe that in the (local) coordinate system $(y_j, t_j)$ on $M^0$, we have $\frac{\partial}{\partial  y_j} = J X_j$ and $\frac{\partial}{\partial t_j} = X_j$. It thus follows that $(y_j + \sqrt{-1} t_j)$ is a $J$-holomorphic local coordinate system.  By Corollary~\ref{surjectivegradient}, we know that when we consider $y$ as a function of $x$, it maps $\mathring{\Pol}$ diffeomorphically to $\R^n$. Furthermore, it follows from the boundary condition of $u$ (near $v=0$) given in Proposition~\ref{compactification} that  $y_j = \frac{1}{2} x_j + \mathrm{smooth}$. Thus, 
\[ z_j := e^{y_j + \sqrt{-1} t_j} \]
 defines a $(\C^{\times})^n$-equivariant map $\Phi_v : (\C^{\times})^n \to M^0$. A careful inspection of the proof of Proposition~\ref{biholomorphic} detailed in \cite[Lemma 2.14]{cifarelli-uniqueness} yields that, in fact, $\Phi_v$ extends to a $(\C^{\times})^n$-equivariant chart  $\Phi_v : \C^n \to M$,  sending the origin to $v$ and each coordinate hyperplane $\{z_i=0\}$ to the pre-image of the facet $\{L_j(x)=0\}$. 
 
Notice that when $\Pol$ is the standard cone and $(M, J_0)=\C^n$, see Example~\ref{e:Rn}, the map $\Phi_0$ constructed as above is the biholomorphism $\beta$ needed. \end{rem}

\subsection{A criterion}
Below, and throughout the text, positive constants $C, C', C''$ etc. may vary from line to line.
\begin{lemma}\label{criterion} Suppose $u$ is a strictly convex smooth function on $\mathring{\Pol}$ satisfying the conditions of Proposition~\ref{compactification} and let $(g, J)$ be the corresponding K\"ahler structure. Suppose furthermore that $g$ is a complete riemannian metric and for any fundamental vector field $X_v, \, v\in \tor$, 
\begin{equation}\label{linear}
||X_v||_g \le C d^g(p_0, \cdot) + C,
\end{equation} where $C>0$ is a constant, $p_0$ is a fixed point of $M$ and $d^g$ is the metric distance of $g$. Then $(M, J)$ is $\T^n$ equivariantly biholomorphic to $(M, J_0)$.
\end{lemma} 
\begin{proof}We claim that for each fundamental vector field $X_v$, the vector field $JX_v$ is complete, and therefore the lemma follows from an application of Proposition \ref{biholomorphic}.  To see this, first note that $||JX_v||_g = ||X_v||_g$. Let $\gamma:[0, T) \to M$ be any integral curve of $JX_v$ and denote $\gamma(0) = q$. We claim that $\gamma$ extends past time $T$.  We consider the quantity $\log\left( \ell_\gamma(t) +1 \right)$, where $\ell_\gamma(t) = \int_0^t ||JX_v||_g ds$ is the length of $\gamma$. By \eqref{linear} we have
\begin{equation*}
\begin{split}
\left| \frac{\partial }{\partial t} \log\left( \ell_\gamma(t) + 1 \right) \right| &= \frac{ ||JX_v||_g (t)}{\ell_\gamma(t) + 1} \leq C \frac{d^g(p_0, \gamma(t)) + 1}{d^g(q, \gamma(t)) + 1} \leq C',
\end{split}
\end{equation*}
for a constant $C'$ depending on C in \eqref{linear} and $q$. It follows that $\log(\ell_\gamma(t) + 1)$ is uniformly Lipschitz on $[0,T)$, and so 
\[ \left| \log(\ell_\gamma(t) + 1)  \right| \leq C'|T - 0| \leq C' T.\]
In particular, we have that 
\[d^g(p_0, \gamma(t)) \leq d^g(p_0, q) + \ell_\gamma(t) \leq C' e^{C'T},  \]
so the result follows from the completeness of $g$. 
\end{proof}

\subsection{Toric steady gradient K\"ahler-Ricci solitons}  
 By definition, a \emph{steady gradient K\"ahler-Ricci soliton} (steady GKRS) on $M$ is a K\"ahler structure $(J, g, \omega)$ together with a real $J$-holomorphic vector field $X$ such that \eqref{steadyKRS} is satisfied 
and $X=  J\nabla f= -\omega^{-1}(df)$ for a function $f \in C^\infty(M)$. The latter condition means that $f$ is a Killing potential of $X$ and \eqref{steadyKRS} becomes
\[ \rho= \frac{1}{2} d d^c f.\]
 We now consider the case when $(M, J, g, \omega)$ is a toric K\"ahler manifold, which is also a steady GKRS with respect to the vector field $X=-\omega^{-1}(df)$. By averaging $f$ over $\T$, we can assume that %~\footnote{{\color{blue} VA: I am not sure what ``without loss'' means here: we are changing the soliton vector field, and therefore the soliton itself} \textcolor{red}{CC: Let $\bar{f}$ be the average of $f$ over $\T$, and let $\overline{X} = -\omega^{-1}(d\bar{f})$. Then $\mathcal{L}_{\overline{X} - X} \omega = dd^c(\bar{f} - f) = 0$, so $\overline{X} = X + K$, where $K$ is some Killing field on $M$. The soliton equation is clearly invariant by replacing $X$ by $X$ plus a Killing field, so ``without loss'' here means up to this natural action. }} 
 $f$ is $\T$ invariant, perhaps after modifying the initial vector field by the addition of a Killing field. In this case, $f$ is necessarily the pullback by the moment map of a linear function $\langle b, x \rangle, \, b\in \tor$ (see for example \cite[Lemma 3.1]{apostolov-notes} or \cite[Section 2.3]{cifarelli-uniqueness}), so that $X=X_b$ is the fundamental vector field of $b$. In this case, we have (see ~\cite{WangZhu, Donaldson-survey} for a similar discussion in the Fano case)
\begin{lemma}\label{l:toric-GKRS} Suppose there exist constants $b \in \tor$, $c\in \tor^*$ and  a symplectic potential $u(x)$ satisfying the conditions of Proposition~\ref{compactification}, such that (in some basis of $\tor$)
\begin{equation}\label{rMA2}
	e^{\langle b, \, x \rangle} = e^{\langle c,  \nabla u\rangle}\det \left({\rm Hess}(u)\right).
\end{equation}
Then $(M, \omega_0, J_u, X_b)$ gives rise to a steady GKRS.
\end{lemma}
\begin{proof} We use a basis of $\tor$ to identify $\tor \cong \R^n$ and $\tor^* \cong (\R^n)^*$.
 Write $y:=\nabla u = \sum_{j=1}^n\left(\frac{\partial u} {\partial x_j}\right) e_i$ in this basis.  Applying Proposition~\ref{legendretransform}, we see by \eqref{rMA2} that there is a K\"ahler potential $H(y)$ on $M^0$ such that
\begin{equation}\label{e:rMA1}
   e^{ -J_u X_b \cdot H} \det {\rm Hess}(H) = e^{\langle c, \, y \rangle}. 
\end{equation}
Using that  $y=\nabla u$ are pluriclosed coordinates, 
 \[{\rm Hess}(H)_{ij} =\left({\rm Hess}(u)\right)^{-1}_{ij}= g(X_{e_i}, X_{e_j})\]
 (by the properties of Legendre transform),  and  taking $dd^c\log$ of both sides of \eqref{e:rMA1} shows that $(J_0, \omega)$ defines a steady GKRS on $M^0$ with respect to $X_b$, and therefore the result follows by continuity. 
\end{proof}

\begin{dfn} A steady GKRS given by the construction of Lemma~\ref{l:toric-GKRS} will be referred to as a \emph{toric} steady GKRS.
\end{dfn}

\begin{lemma}\label{l:rMA1} Let $(M, \omega_0)$ be the symplectic toric manifold determined by a Delzant polyhedron $(\Pol, \Lambda)$. Then, 
there exists at most one constant $c \in \tor^*$ for which \eqref{rMA2} can admit a solution $(u, b)$ with $b \in \tor$ and $u(x)$ satisfying the conditions of Propositions \ref{compactification}. If it exists, the constant $c$ is determined by the condition
\begin{equation}\label{CY}
 \langle dL_j, c \rangle =2, \end{equation}
for any label $L_j(x)$ of $(\Pol, \Lambda)$.
\end{lemma}
\begin{proof}  Taking the logarithm of \eqref{rMA2}, we get 
\[ \log \det {\rm Hess}(u) +  \sum_{i = 1}^n c_i \frac{\partial u}{\partial x_i} = \langle b, \, x \rangle. \]
Clearly the right hand side is smooth on $\Pol$. Moreover it is not hard to see using Proposition \ref{compactification} that, modulo $C^\infty(\Pol)$, the left hand side is given by
\[-\sum_{j = 1}^d \left(1 - \frac{1}{2}\langle dL_j, c \rangle \right) \log L_j(x),\] 
so that  we must have $\langle dL_j, c \rangle = 2$ for any label $L_j$.  As $\Pol$ is a Delzant polyhedron with at least one vertex,  there are at least $n = \dim_\R \tor$ linearly independent $dL_i$, so we see that $c$ is uniquely determined, should it exist. \end{proof}

\begin{rem} The proof of Lemma~\ref{l:rMA1} shows that \eqref{CY} is a necessary and sufficient condition for the existence of a globally defined Ricci potential  with respect to any K\"ahler metric $(\omega_0, J_u)$ associated to a symplectic potential  $u$ satisfying the conditions of Propositions~\ref{compactification}. In this case, the corresponding Ricci potential is the smooth extension of the function
 \[ \log \det {\rm Hess}(u) +  \sum_{i = 1}^n c_i \frac{\partial u}{\partial x_i}.  \]
The condition \eqref{CY} also implies  the existence of a global holomorphic volume form on $(M, J_0)$ (see e.g. \cite[Prop.~6.8]{FOW}).  Indeed,  in the basis of $\Lambda$ associated  to a vertex $v$ of $\Pol$,  \eqref{CY} yields that $c=(2, \ldots, 2)$.   In the holomorphic coordinates $(z_1, \dots, z_n)$ on $M^0\cong (\C^{\times})^n$ furnished by Proposition~\ref{biholomorphic},  we consider the holomorphic function $z^{\frac{c}{2}}:= z_1\dots z_n$. By \eqref{CY}, in the equivariant holomorphic chart $(z_1^w, \ldots, z_n^w)$ of $M$ associated to any other vertex $w$ of $\Pol$, we have $z^{\frac{c}{2}}= C_w z_1^{w}\cdots z_n^{w}$, showing that $z^{\frac{c}{2}}$ extends to a global holomorphic function on $(M, J_0)$ vanishing at order one along  the  toric ``divisor''  corresponding to the pre-images of each facet $F_j$ of $\Pol$. The product of this with  the meromorphic volume form $\left(X_1^{1,0} \wedge \cdots \wedge X^{1,0}_n\right)^{-1}$ of $(M, J_0)$ (where $X_1, \ldots, X_n$ are the fundamental vector fields on $M$ corresponding to any  lattice basis of $\tor$) is therefore a non-vanishing section of the canonical bundle $K_{M}$.
\end{rem}
By \cite[Proposition 2.15]{cifarelli-uniqueness}, any K\"ahler metric $\omega$ constructed via \eqref{toric} and satisfying the conditions of Proposition \ref{compactification} lies in the same cohomology class as the reference metric $\omega_0$. Thus combining the above, we obtain
\begin{prop}\label{rMAprop} Suppose $(M, J_0, \omega_0)$ is a toric manifold corresponding to a Delzant polyhedron $(\Pol, \Lambda)$ which satisfies \eqref{CY} for a unique point $c\in \tor^*$. Then any solution  $(u, b)$ of \eqref{rMA2} where $b\in \tor$ and $u(x)$ satisfies the conditions of Propositions~\ref{compactification} and \ref{biholomorphic} gives rise to a  toric steady GKRS  $(\omega, X_b)$ on $(M, J_0)$,  in the deRham  cohomology class of $[\omega_0]$.
\end{prop}

\subsection{The soliton vector field}
As before, let $\Pol$ be a polyhedron with labels $L_1, \dots, L_d$.  Denote by $\bar{L}_j$ the linear part of $L_j$, so that $\bar{L}_j(0) = 0$. We define the \emph{recession cone} $C(\Pol)$ to be the cone  
\begin{equation}\label{recession-defining-eqns}
	C(\Pol) = \{ x \in \tor^* \: | \: \bar{L}_j(x) \geq 0\}, 
\end{equation}
and its \emph{dual}
\begin{equation}\label{dual-defining-eqns}
	C^*(\Pol) = \{ y \in \tor  \: | \: \langle y ,\, x \rangle \geq 0 \textnormal{ for all } x \in C(\Pol) \}.
\end{equation}

\begin{lemma}[Forbidden region lemma]\label{frlemma}
Suppose that $(M, J_0, \omega)$ is a steady GKRS given by Proposition~\ref{rMAprop}, so that the soliton vector field $X = J_0\gr_g f$ is the fundamental vector field associated to $b \in \tor$. Then 
\[ \int_M e^{f} \omega^n = \infty,  \] 
and consequently $b \not\in -C^*(\Pol)$. 
\end{lemma}
\begin{proof}
By \eqref{e:rMA1}, we have that 
\[ \int_M e^{f} \omega^n =(2\pi)^n \int_{\R^n} e^{\langle c, \, y \rangle} dy = + \infty.\] 
By \eqref{rMA2}, we have that the right hand side above is equal to a multiple of
\[ \int_\Pol e^{\langle b, \, x \rangle} dx. \]
It's clear that $e^{\langle b, \, x \rangle}$ is integrable on $\Pol$ if and only if $b \in -C^*(\Pol)$.
\end{proof}
\begin{rem}
In Lemma~\ref{frlemma},  we do not assume anything a priori about the curvature of $(J_0, \omega)$.  In general, without the toric assumption, one can show that a steady GKRS has infinite weighted volume under assumptions on the scalar curvature decay at infinity or on the positivity of Ricci curvature (see for example \cite{bryant, ConDer}). \end{rem}

\section{Cao's steady GKRS  on $\mathcal{O}(-n)$ and $\C^n$}\label{s:calabi-type} 
\subsection{Complete steady GKRS's  on $\cO(-n)$} We recall here the construction of complete steady GKRS on \[M= \cO(-n) \to \PP^{n-1}, \] discovered by Cao~\cite{Cao};  it is also observed  (cf.~\cite{FW, Schafer1}) that the construction can be extended to suitable powers of the canonical bundle of a Fano K\"ahler--Einstein manifold. We use here the Calabi ansatz from the point of view of Hwang--Singer~\cite{HS},  but following the notation of \cite{ACG} as a warm-up for the more general construction we present in the next sections.

\smallskip 
The K\"ahler metric is taken of the following form, known  as the {\it Calabi ansatz}:
\begin{equation}\label{metric}
\begin{split}
g &= \xi \check{g}_{\rm FS} + \frac{d\xi^2}{\Theta(\xi)} + \Theta(\xi) \theta^2 \\
\omega &= \xi  \check{\omega}_{\rm FS} + d\xi \wedge \theta,
\end{split}
\end{equation}
where:   
\begin{enumerate}
\item[$\bullet$] $\check{\omega}_{\rm FS}$ is a Fubini--Study metric on $\PP^{n-1}$ of scalar curvature \newline $\check{s}_{\rm FS}= \frac{2n(n-1)}{r}$,
\item[$\bullet$] $\theta$ is a connection 1-from with $d\theta = \check{\omega}_{\rm FS}$,
\item[$\bullet$] $\Theta(\xi)$ is a positive smooth function (called a {\it profile function} in \cite{HS}), defined on the interval $(\alpha, \infty)$ for some $\alpha>0$, and satisfying
\begin{equation}\label{calabi-boundary}
\Theta(\alpha)=0, \ \ \Theta'(\alpha)= 2.
\end{equation}
\end{enumerate}
The induced complex structure $J=g^{-1}\omega$ is invariant under the $\Sph^1$ action on the fiber, acts as the pullback of the complex structure on $\PP^{n-1}$ on the horizontal distribution, and satisfies 
\begin{equation}\label{J1}
J d\xi  = \Theta(\xi) \theta.
\end{equation}

For any such profile function $\Theta(\xi)$, the corresponding metric is well-defined on the total space $M^{0} \cong \cO(-r)^{\times}$ of the principal $\mathbb C^*$-bundle of degree $r$  over\footnote{Our normalization for the scalar curvature of $\check{\omega}_{\rm FS}$  yields that $[\check{\omega}_{FS}] = r \gamma$ where $\gamma$ is the generator of $H^2(\PP^{n-1}, \mathbb Z) \cong \mathbb Z$.} $\PP^{n-1}$ whereas the boundary conditions \eqref{calabi-boundary} imply that  the metric smoothly extends to the total space $M= \cO(-r)$ (see e.g.  \cite{HS,ACGT}). 

The completeness of the metric is studied in \cite{HS}, and it is shown that it is equivalent to the growth of $\Theta(\xi)$ at $\infty$ being at most quadratic in $\xi$.

In general, $J$ is not biholomorphic to the standard complex structure of $M$ (as this might not be true fibre-wise). For this to be true, one needs to ensure (see Proposition~\ref{biholomorphic}) that the vector field $JX$ is complete, where $X$ is the generator of the $\Sph^1$-action on the $\R^2$-fibre of $M$; the latter can be checked to hold true via Lemma~\ref{criterion} if $\Theta(\xi)$ grows at $\infty$ at most linearly in $\xi$.

A Ricci potential  $\kappa$ for the metric \eqref{metric}, written on $M^{0}$, is given  by the 
formula \cite{ACG,HS}:
\begin{equation}\label{scal}
\kappa = \check{\kappa}_{\rm FS} -\frac{1}{2} \log\left(\xi^{n-1} \Theta(\xi)\right),
\end{equation}
where, $\check{\kappa}_{\rm FS}$ is a local Ricci potential of $\check{\omega}_{\rm FS}$.
Furthermore, the function
\begin{equation}\label{affine-linear}
a\xi
\end{equation}
is a Killing potential for $(g, J, \omega)$, and  the function
\begin{equation}\label{relative-PH}
 y (\xi) := \int^{\xi} \frac{d\xi}{\Theta(\xi)} 
 \end{equation}
satisfies (see \eqref{J1})
\[ dd^c y= d\theta = \check\omega_{\rm FS}. \]
 We thus obtain the following ansatz for constructing steady GKRS:
 \begin{lemma}\label{KRS-Calabi} If $\Theta(\xi)$ satisfies
 \[\frac{n}{r} \int^{\xi} \frac{d\xi}{\Theta(\xi)}  -\frac{1}{2} \log\left(\xi^{n-1} \Theta(\xi) \right)= a\xi, \]
 for some real constants $a$,  then \eqref{metric} describes a steady GKRS  with corresponding Killing potential $2a\xi$. %~\footnote{{\color{blue} VA: Un fortunately, in the rest of the paper we use a different convention for $X$: we  are solving $\rho=dd^c f$ instead of $\rho = \frac{1}{2}dd^c f$}.}
 \end{lemma}
Taking one derivative in the relation above, we get a description of steady GKRS in terms of solutions of a linear ODE~\cite{FIK, FW}:
\begin{equation}\label{ODE-KRS}
\Theta'(\xi) + \left(2a+\frac{(n-1)}{\xi}\right) \Theta(\xi) -\frac{2n}{r}=0.
\end{equation}
The general solution of the ODE 
\[ \psi'(\xi) + p(\xi) \psi(\xi) -q(\xi)=0 \] is given by
\begin{equation}\label{general-solution}
\psi(\xi)= e^{-\int^{\xi} p(x) dx}\left(\int^{\xi} q(x) e^{\int^x p(\tau)d\tau} dx + C\right).
\end{equation}
Specializing to \eqref{ODE-KRS} with $a\neq 0$,  we get the following explicit form of the unique solution $\Theta_{\alpha, a}(\xi)$ of \eqref{ODE-KRS} satisfying $\Theta_{\alpha,a}(\alpha)=0$:
\begin{equation}\label{Futaki-soliton}
\begin{split}
\Theta_{\alpha,a}(\xi) &=  \left(\frac{2n}{r}\right) e^{-2a\xi}\xi^{-n+1}\left(\int_{\alpha}^{\xi} e^{2ax}x^{n-1} dx\right) \\
                &=  \left(\frac{2n}{r}\right) \left[\frac{(2a\xi)^{-n+1}}{a} \left(\sum_{k=0}^{n-1} (-1)^{n-k-1} \frac{(n-1)!}{k!}(2a\xi)^k\right) \right. \\  
                 & \hspace{2.4in} \left. \phantom{\sum_{1}^2} - C_{\alpha}e^{-2a\xi}(2a\xi)^{-n+1}\right], \\
             C_{\alpha} &:=   \frac{e^{2a\alpha}}{a}\left(\sum_{k=0}^{n-1} (-1)^{n-k-1} \frac{(n-1)!}{k!}(2a\alpha)^k\right).     \end{split}
\end{equation}
The second boundary conditions $\Theta'_{\alpha, a}(\alpha)=2$ forces $n=r$, i.e. a global smooth solution can only exist on $\cO(-n)=\cK_{\PP^{n-1}}$. We conclude from the above formula that $\Theta_{\alpha,a}(\xi)$ is positive on $(\alpha, \infty)$ and, as long as $a>0$, grows at a rate $O(1)$ as $\xi \to \infty$. Therefore the metric is complete and the induced complex structure is $\C^*$-equivariantly biholomorphic to $\cO(-n)$.  We thus get a \emph{complete} steady GKRS on $\cO(-n)$ for any $a>0$ and $\alpha>0$.  The first line of the formula also makes sense for $a=0$, in which case the solution is
\[\Theta_{\alpha, 0}(x)=  \frac{2}{\xi^{n-1}}\left(\frac{\xi^{n}}{n} - \frac{\alpha^{n}}{n}\right).\]
It has a linear growth in $\xi$ at infinity, and corresponds to the multi-dimensional Eguchi--Hanson ALE Ricci-flat  metric on $\cO(-n)$~\cite{EH,LeBrun}.

\begin{rem}\label{r:cigar} The above ansatz also applies to the case $n=1$, when $\PP^{n-1}$ is a point and $\cO(-1)=\C$. In this case, the solution of \eqref{ODE-KRS} is 
\[\Theta_{\alpha, a}(\xi) = \frac{1}{a}\left( 1 -e^{-2a(\xi-\alpha)}\right).\]
The parameter $\alpha$ here is  coming from a translation of the momenta $\xi$ (and can be set $\alpha=0$). When $a>0$, this corresponds to the cigar soliton~\cite{Hamilton}.
\end{rem}

\begin{rem}\label{r:homothety} We got a two parameter family of   GKRS's on $\cO(-n)$, parametrized by the positive constants $\alpha>0$ and $a>0$. There is an equivalence relation 
\[  \frac{1}{\lambda}\Theta_{\alpha, a} (\lambda \xi) = \Theta_{\lambda \alpha, \frac{a}{\lambda}}(\xi), \qquad \lambda >0, \]
on the set of solutions. Geometrically, this corresponds to  the $\R_+$-action by a multiple of $\lambda$ on the fiber $\cO(-n)$,  and a homothety by a factor of $\frac{1}{\lambda}$ on the metric.  Taking quotient by this equivalence, we get a $1$-parameter family of steady GKRS's,  realizing \emph{any} positive multiple $a$ of the generator $-JX$ of the $\R_+$-action on $\cO(-n)$. For all these solutions, the profile function  has the asymptotic  at infinity
\begin{equation}\label{a-asymptotic}
\Theta_{\alpha, a}(\xi) = \frac{1}{a} + O(\xi^{-1}). \end{equation}
\end{rem}

\subsection{Complete steady GKRS's on $\C^{n}$}\label{s:Cao} The construction in the previous section also produces complete irreducible GKRS's on $\C^{n}$ for $n>1$, originally due to  Cao \cite{Cao}. We recast here these examples  by using the so-called \emph{blow-down} Calabi ansatz, see Koiso--Sakane~\cite{KS} in the compact case. We use the formalism of \cite{ACGT}, see in particular Proposition 2 there. 

Let $r=1$, i.e. $M=\cO(-1) \to \PP^{n-1}$,  and put  $\alpha=0$ in \eqref{calabi-boundary}.  It was observed in \cite{KS} that \eqref{metric} still define smooth tensors on $\cO(-1) \to \PP^{n-1}$. These tensors  give rise to a K\"ahler structure on $\cO(-1)\setminus S_0$ (as they  degenerate along the zero section $S_0$ of $\cO(-1)$), but they do describe the pull-back to $\cO(-1)$ of a smooth K\"ahler metric defined on the \emph{blow-down} $\C^{n}$  of $\cO(-1) \to \PP^{n-1}$ along $S_0$. In this case, the only solutions of \eqref{ODE-KRS} (with $r=1$) which satisfy \eqref{calabi-boundary}  are
\begin{equation}\label{Cao-soliton}
\Theta_a(\xi)= 2e^{-2a\xi}\xi^{-n+1}\left(\int_{0}^{\xi}e^{2ax} x^{n-1} dx\right).
\end{equation}
We again observe that if $a>0$ the asymptotic at infinity of $\Theta_a(x)$ is
\[ \Theta_a(x) = \frac{1}{a} + O(\xi^{-1}), \]
thus showing the metric is complete and the induced complex structure is biholomorphic~\footnote{A detailed argument in a greater generality is provided in Lemma~\ref{l:biholomorphic}} to $\C^n$. When $a=0$, we get the solution $\Theta_0(\xi)=2\xi$, which corresponds to the flat metric on $\C^{n}$ (see e.g. \cite{ACG}).

\begin{rem}\label{r:Cao-soliton-vector} There is an orbifold holomorphic map
$$B : {\mathcal O}(-r) \to {\mathbb C}^{n}/{\Gamma}_{r},$$
where $\Gamma_{r} \subset {\rm U}(n)$ is the cyclic group of order $r$, generated by  $\varepsilon_{r} {\rm Id}$ for a primitive $r$-th root of unity $\varepsilon_{r}$,  
defined as follows.  Let $[x_1,  \ldots,  x_n]$ be a point on $\PP^{n-1}$ written in  homogeneous coordinates and denote by $(x_1, \ldots, x_n)^{\otimes r}$ the corresponding generator of the fibre of ${\mathcal O}(-r) \to \PP^{n-1}$. Then,  for $x=([x_1, \ldots, x_n]; \lambda (x_1, \ldots, x_n)^{\otimes r}) \in \cO(-r)$,  
$$B(x):= \lambda^{\frac{1}{r}} (x_1, \ldots , x_n).$$
The map $B$ provides a biholomorphism between $M^{0}:= \mathcal{O}(-r)\setminus {S}_0$ (where $S_0$ stands for the zero section) and $({\mathbb C}^{n}\setminus \{0\}) /{\Gamma}_{r}$, with inverse 
$$B^{-1} (z_0, \ldots z_n)= ([z_1, \ldots,  z_n]; (z_1, \ldots , z_n)^{\otimes r}).$$
Note that $B$ contracts the zero section ${S}_0$ of $\cO(-r)$ to the origin of ${\mathbb C}^{n}/{\Gamma}_{r}$, so it can be thought as an orbifold blow-down map.

Specializing the above picture to $r=1$, we see that  Cao's soliton vector field  is given by  $2a$-times the generator of the $\R_+$-action on $\cO(-1)\setminus S_0$, which corresponds to the vector field $ 2a\left(\sum_{i=1}^n r_i \frac{\partial}{\partial r_i}\right)$ on $\C^{n}$.  
Up to the homothety action on $\cO(-1)\setminus S_0$ (which corresponds to the homothety action on $\C^n$), and rescaling of the metric, we can assume that  $2a=1$.

One can also produce steady GKRS solitons on the orbifold $\C^{n}/\Gamma_\ell$,  by considering \eqref{Futaki-soliton} with $\alpha=0$ on $\cO(-\ell)^{\times}$. Such a GKRS pulls back to $\C^n$ to define a complete GKRS on $\C^n$ with soliton vector 
field $ \frac{2a}{\ell}\left(\sum_{i=1}^n r_i \frac{\partial}{\partial r_i}\right)$. However, by virtue of Remark~\ref{r:homothety}, the induced GKRS on $\C^n$ is homothetic to a GKRS in the $1$-parameter family obtained from $\cO(-1)^{\times}$.

\end{rem}

\subsection{Toric description of Cao's GKRS} We now recast the steady GKRS examples on $\cO(-n)$ and on $\C^n$ in terms of the  formalism of Section~\ref{s:toric}.  This also provides a useful link with Bryant's work~\cite{bryant}.

We consider more generally a K\"ahler structure of the form \eqref{metric},  where $\PP^{n-1}$ is replaced by any \emph{toric} compact K\"ahler--Einstein $(n-1)$-manifold $(\check{M}, \check{\omega}_{\check M}, \check{\T}^{n-1})$, such that $\frac{1}{2\pi}[\check{\omega}_{\check{M}}] \in H^2(\check{M}, \Z)$ and $2\pi c_1(\check{M})= r \check{\omega}_{\check{M}}$ for some $r\in \Z_+$; in this case,  \eqref{metric} describes a smooth K\"ahler metric $\omega$ defined on  $M^{0}=\cL^{\times} \to \check{M}$ with $\cL^r = \cK_{\check{M}}$; when $\alpha>0$ it extends smoothly to $M=\cal L$.  We note that the K\"ahler structure \eqref{metric} is toric, under a lifted action of the torus $\check{\T}^{n-1}$ to $\cL$,  and the natural $\Sph^1$-action on $\cL$ by multiplication on the fibre. More precisely, \cite[Lemma 6]{ACGL} tells us that
\begin{itemize}
\item if $\check{\mu} : (\check{M}, \check{\omega}_{\check M}, \check{\T}^{n-1}) \to \check \tor^*$ is a momentum map of the toric K\"ahler--Einstein $(n-1)$-manifold, then 
\begin{equation}\label{momenta}
\mu :=(\xi, \xi \check{\mu})
\end{equation}  is a $\T^n$-momentum map for the induced metric \eqref{metric} on $\cL$;
\item if $\check{u}(\check{\mu})$ is a symplectic potential of the K\"ahler--Einstein metric on $\check{M}$, then 
\begin{equation}\label{calabi-symplectic-potential}
U(\mu):= u(\xi) + \xi \check{u}(\check{\mu})\end{equation} is a symplectic potential of \eqref{metric}, where $u''(\xi)= \frac{1}{\Theta(\xi)}$.
\end{itemize}
Formula \eqref{momenta} determines the momentum image  of $(M, \omega)$ in terms of the momentum image of $\check{\mu}$ whereas \eqref{calabi-symplectic-potential} can be used to determine the corresponding labels of the polyhedron $\mu(M)$. 

\smallskip
Specifically, let us first consider the case $M=\cO(-n) \to \PP^{n-1}$. The normalization $[\check{\omega}] \in 2\pi c_1(\PP^{n-1})$ tells us that the corresponding labelled Delzant polytope $(\check{\Pol}, \check{\La})$  is
%\textcolor{red}{\[ \begin{split} (\check{\Pol}, \check{\La})= \left\{ x \in \R^{n-1} \,  \Big|\right. & \, \check{L}_i(x):=\big(x_i + 1\big) \ge 0,   \\ 
%  & \check{L}_0(x):=\big(1-(x_1+ \cdots +x_{n-1})\big) \ge 0, \\
%  & \hspace{1.5in}\left.i=1, \ldots n-1 \phantom{\Big|} \right\}. \end{split}\]}
 \[  (\check{\Pol}, \check{\La})= \left\{ x \in \R^{n-1} \left| \, \begin{array}{l} \check{L}_i(x):=\big(x_i + 1\big) \ge 0, \\  \check{L}_0(x):=\big(1-(x_1+ \cdots +x_{n-1})\big) \ge 0, \\  i=1, \ldots n-1  \end{array} \right.   \right\} \]
Using \eqref{momenta} and letting 
\begin{equation}\label{cao-mu-i}
 \mu_i:= \xi\check{\mu}_i + \xi, \, \,  i=1, \ldots, n-1, \qquad  \mu_n:=\xi-\sum_{i=1}^{n-1} \xi \check{\mu}_i, 
\end{equation}
we see that the image $\mu(M)$ of $(M, \omega)$ in these momentum coordinates is given by
%\[ \Pol = \left\{ \mu \in \R^n \,  \Big| \,  L_i(\mu): = \mu_i\ge 0, \,  L_0(\mu):= \left(\frac{1}{n}\sum_{i=1}^{n}\mu_i - \alpha\right)\ge 0, \,   i=1, \ldots, n \right\}. \]
\[ \Pol = \left\{ \mu \in \R^n \,  \left|\, \begin{array}{l} L_i(\mu): = \mu_i\ge 0, \\  L_0(\mu):= \left(\frac{1}{n}\sum_{i=1}^{n}\mu_i - \alpha\right)\ge 0, \\  i=1, \ldots n  \end{array} \right. \right\}. \]
Using that (see \cite{guillemin})  \[\check{u}(\check{\mu})= \frac{1}{2}\sum_{i=1}^{n-1} \check{L}_i(\check{\mu}) \log \check{L}_i(\check{\mu})\]
is a symplectic potential of $\check{\omega}_{\rm FS}$ and (see \cite{abreu0}) 
\[ u(\xi) = \frac{1}{2}(\xi -\alpha) \log(\xi-\alpha) + {\rm smooth},\] we obtain from \eqref{calabi-symplectic-potential} 
together with the fact that $\sum_{i=1}^n \mu_i = n\xi$ by \eqref{cao-mu-i}, that
\[U(\mu_1, \ldots, \mu_n) = \frac{1}{2}\sum_{i=1}^n L_i(\mu) \log L_i(\mu) +  \mathrm{smooth}. \]
The above asymptotic shows that
\[(\Pol, \La), \qquad \La:=\{L_0, \ldots, L_n\}\]
is the corresponding labelled Delzant polyhedron. In this description, it follows from \eqref{cao-mu-i} that the Killing potential of the GKRS vector field of $g_{\alpha, a}$ becomes 
\[ 2a \xi =  \frac{2a}{n}\left(\mu_1 + \cdots + \mu_n\right).\]

\bigskip
In the case of $\C^n$  ($n\ge 2$) we proceed in a similar way, but starting with the labelled polytope
%\[(\check{\Pol}, \check{\La})= \left\{ x \in \R^{n-1} \,   \Big|  \, \check{L}_i(x):=\left(x_i + \frac{1}{n}\right) \ge 0, \,  \check{L}_0(x):=\left(\frac{1}{n}-(x_1+ \cdots +x_{n-1})\right) \ge 0, \, i=1, \ldots n-1 \right\}\]
\[(\check{\Pol}, \check{\La})= \left\{ x \in \R^{n-1} \,   \left|\, \begin{array}{l} \check{L}_i(x):=\left(x_i + \frac{1}{n}\right) \ge 0, \\   \check{L}_0(x):=\left(\frac{1}{n}-(x_1+ \cdots +x_{n-1})\right) \ge 0, \\  i=1, \ldots n-1  \end{array} \right.\right\}\]
corresponding to the different normalization $[\check{\omega}_{\rm FS}] \in \frac{2\pi}{n} c_1(\PP^{n-1})$. Letting
\[ \mu_i:= \xi\check{\mu}_i + \frac{\xi}{n}, \, \,  i=1, \ldots, n-1, \qquad  \mu_n:=\frac{\xi}{n}-\sum_{i=1}^{n-1} \xi \check{\mu}_i, \]
we identify the labelled Delzant polyhedron of the GKRS metric $g_a$ on $\C^n$ to be
\[ (\Pol_n, \La) = \left\{ \mu \in \R^n \,  \Big| \,  L_i(\mu): = \mu_i\ge 0,  \,   i=1, \ldots, n \right\}. \]
The soliton vector field of $g_a$ is determined by the Killing potential
\[  2a(\mu_1+ \cdots +  \mu_n).\]

\section{Steady GKRS with a hamiltonian $2$-form}\label{s:hfkg} We recall here the description in \cite{ACG} of the K\"ahler structures admitting a non-trivial hamiltonian $2$-form of order $\ell$. We shall use this construction as an ansatz for constructing  steady GKRS's, extending the constructions via the Calabi ansatz discussed in Section~\ref{s:calabi-type}. 

\subsection{K\"ahler metrics with hamiltonian 2-forms} Let $(M, J, g, \omega)$ be a K\"ahler manifold and $\nabla$ the corresponding Levi--Civita connection.
\begin{dfn}\label{d:hamiltonian-form} A \emph{hamiltonian $2$-form} $\phi$ on $(M, g, J, \omega)$ is a  form of type $(1,1)$ on $(M, J)$ which satisfies 
\[ \nabla_X\phi = \frac{1}{2}\left(d \tr_{\omega} \phi \wedge g(JX)+ (J^*d\tr_{\omega}\phi) \wedge g(X)\right)  \]
for all vector fields $X$, where $\tr_{\omega}\phi:=\langle \phi, \omega\rangle_g$ denotes the trace of $\phi$ with respect to $\omega$. The hamiltonian form is \emph{non-trivial} if $\phi\neq 0$.
\end{dfn}

It turns out that the existence of a non-trivial hamiltonian $2$-form $\phi$ on $(M, J, g, \omega)$, which is not a multiple of $\omega$,  imposes  very strong conditions on the K\"ahler manifold.  

First (see \cite[Prop.~13]{ACG}) the $n$ elementary symmetric functions of the $n$ real eigenvalues of the Hermitian  symmetric operator $\phi \circ \omega^{-1}$  on $(TM, J, g)$ are $\omega$-hamiltonian potentials of Killing vector fields $K_1, \ldots,$ $K_n$ of $(M, J, g, \omega)$.  Thus, $K_1, \ldots, K_n$ span an abelian Lie sub-algebra $\mathfrak{t}$ in the Lie algebra  of Killing vector fields of $(J, g, \omega)$. The integer
\[\ell=  {\rm dim}(\mathfrak{t})={\rm dim} \left({\rm span}_{\R} \{K_1, \ldots, K_n\}\right)\]
is called the \emph{order} of $\phi$.  

Notice that the order of $\phi$ is zero iff  $\nabla \phi=0$.  Such a parallel hamiltonian $2$-form,  which is not a multiple of $\omega$,   gives rise to a local product structure of $(M, g, J, \omega)$, i.e. around each point  $(M, g, J, \omega) = \prod_{a=1}^N (B_a, J_a, g_a, \omega_a)$ where  ${\rm dim}_{\C}(B_a)$ is corresponding to the multiplicity of a constant eigenvalue of $\phi \circ \omega^{-1}$, seen as a complex endomorphism of $(TM, J)$. The main result \cite[Theorem 1]{ACG} extends this observation to the general case: if $(M, g, J, \omega)$ admits a non-trivial hamiltonian $2$-form of ordre $\ell\ge 1$, then around each point in a open dense subset of $M$, the K\"ahler structure $(g, J, \omega)$ is isometrically embedded in the following construction of a K\"ahler manifold:

\smallskip
We let $(B, g_B, \omega_B):= \prod_{a=1}^N(B_a, \check{g}_a, \check{\omega}_a)$ be the product of K\"ahler manifolds and  $\eta_a$ real numbers. Suppose there exists a principal $\ell$-dimensional torus  bundle $\T^{\ell} \hookrightarrow P\to B$ and an $\R^{\ell}$-valued connection $1$-form $\theta:=(\theta_1, \ldots, \theta_{\ell})$ on $L$ satisfying
\begin{equation}\label{theta}
d\theta_r = \sum_{a=1}^N (-1)^r\varepsilon_a\eta_a^{\ell-r}\check{\omega}_a.\end{equation}
Fix  real numbers
\[ -\infty \le \alpha_1< \beta_1\le \alpha_2<\beta_2\le \dots \le \alpha_{\ell} < \beta_{\ell}\le +\infty\]
and smooth functions of one variable $F_i(t)$,  defined respectively over the intervals $(\alpha_i, \beta_i)$.  
We  then consider the manifold
\[ M^{0}:= (\alpha_1, \beta_1)\times \cdots \times (\alpha_{\ell}, \beta_{\ell}) \times P\]
and define the following tensors on $M^0$:
\begin{equation}\label{k-order-ell}
\begin{split}
g&=\sum_{a=1}^N \varepsilon_a p_{nc}(\eta_a) \check{g}_a + \sum_{j=1}^{\ell} \left(\frac{p_{c}(\xi_j)\Delta(\xi_j)}{F_j(\xi_j)}\right) d\xi_j^2  \\
	 &\hspace{1.14in}+  \sum_{j=1}^{\ell}\left(\frac{F_j(\xi_j)}{p_c(\xi_j)\Delta(\xi_j)}\right)\left(\sum_{r=1}^{\ell}\sigma_{r-1}(\hat \xi_j)\theta_r\right)^2, \\
\omega &= \sum_{a=1}^N \varepsilon_a p_{nc}(\eta_a) \check{\omega}_a + \sum_{r=1}^\ell d\sigma_r\wedge \theta_r.
\end{split}
\end{equation}
In the above formulae
\begin{enumerate}
\item[$\bullet$] $\varepsilon_a\in \{-1, 1\}, a=1, \ldots, N$ is a sign constant.
\item[$\bullet$] $\xi_i \in (\alpha_i, \beta_i), i=1, \ldots, \ell$ are free variables and $\sigma_r$ (resp. $\sigma_r(\hat \xi_i)$) denotes  the $r$-th elementary symmetric function of $\{\xi_j \}$ (resp. of $\{\xi_j : j\neq i\}$). % we set $\sigma_{\ell+1}= \sigma_{\ell}(\hat \xi_j)=1$.
\item[$\bullet$] $p_{nc}(t):= \prod_{j=1}^{\ell}(t-\xi_j)$ and $p_c(t):= \prod_{a=1}^N (t-\eta_a)^{{\rm dim}_{\C}(B_a)}$.
\item[$\bullet$] $\Delta(\xi_j):= \prod_{i\neq j} (\xi_j - \xi_i)$;
\item[$\bullet$] $\theta_r$ are the components of the connection $1$-form on $P$ defined in \eqref{theta}.
\end{enumerate}
 It is shown in \cite{ACG} that if $\eta_a, \alpha_i, \beta_i$  and $F_i(t)$ are such that  
 \begin{equation}\label{inequality}
	\begin{split}
   &\varepsilon_a p_{nc}(\eta_a)>0 \, \, \textrm{on} \, \, (\alpha_1, \beta_1)\times \cdots \times (\alpha_{\ell}, \beta_{\ell}),  \\
   & \hspace{.2in} (-1)^{\ell-i}F_i(x)p_c(x)>0 \, \, \textrm{on} \, \, (\alpha_i, \beta_i), \end{split} \end{equation}
 then \eqref{k-order-ell} defines a K\"ahler structure  on $M^{0}$ with complex structure given by
 \begin{equation}\label{J2}
 \begin{split}
 Jd\xi_j &= \left(\frac{F_j(\xi_j)}{p_c(\xi_j)\Delta(\xi_j)}\right)\left(\sum_{r=1}^{\ell}\sigma_{r-1}(\hat \xi_j)\theta_r\right), \\
 J\theta_r &= (-1)^r\sum_{j=1}^\ell\frac{{p_c}(\xi_j)}{F_j(\xi_j)} \xi_j^{\ell-r} d\xi_j, 
 \end{split}
 \end{equation}
 and the  lift of the complex structure  $J_B$ on $B$ to the horizontal distribution  of $\theta$. 
 Furthermore, $(g, J, \omega)$ is $\T^{\ell}$-invariant and the smooth functions $\sigma_r,\,  r=1, \ldots, \ell$ are momenta for the $\T^{\ell}$-action. The fundamental vector fields $K_r$ associated to $\sigma_r$ are dual to $\theta_r$ by \eqref{k-order-ell}.
 Finally, the $(1,1)$ form on $(M^0, J)$
 \begin{equation}\label{eq:hamiltonian-2-form}
  \phi = \sum_{a=1}^N \eta_a p_{nc}(\eta_a)\check{\omega_a} + \sum_{r=1}^{\ell}(\sigma_rd\sigma_1- d\sigma_{r+1})\wedge \theta_r
  \end{equation}
 is a hamiltonian $2$-form of order $\ell$ on $(M^0, J, g, \omega)$ such that the eigenvalues of $\phi \circ \omega^{-1}$, seen as a complex operator on  $(TM, J)$, are the smooth functions $\xi_1, \ldots, \xi_{\ell}$ (each of multiplicity $1$)  and the constants $\eta_a$ (each of multiplicity ${\rm dim}_{\C}(B_a)$).
 
  \subsection{The space of hamiltonian $2$-forms} The space $\mathcal{H} (M, J, g, \omega)$ of all hamiltonian $2$-forms on  a given K\"ahler manifold $(M, g, J, \omega)$ is a linear subspace of $\Omega^{1,1}(M,J)$ which contains $\omega$ as a non-zero element.  Furthermore, it is shown  in \cite[Sec.2.2]{ACG} (see in particular Remark 2 in that reference) that $\mathcal{H} (M, J, g, \omega)$ is finite dimensional with ${\rm dim}\left(\mathcal{H} (M, J, g, \omega)\right) \le (n+1)^2$, where $n={\rm dim}_{\C}(M)$. Following the terminology  from \cite{CMR}, we adopt the following
 \begin{dfn} The integer $D(J, g):= \dim\left(\mathcal{H} (M, J, g, \omega)\right)$ is called the mobility of $(M, g, J, \omega)$. In general, $1\le D(J, g) \le (n+1)^2$ and the presence of a non-trivial hamiltonian $2$-form of order $\ell \ge 1$ implies $D(J, g) \ge 2$.
 \end{dfn}
The term \emph{mobility} in the above definition stems from the remarkable fact (observed in \cite{CMR})  that the vector space $\mathcal{H} (M, J, g, \omega)$  of hamiltonian $2$-forms on $(M, g, J, \omega)$ is isomorphic to the vector space generated by the set $Sol(J, g)$ of  smooth $(g, J)$  Hermitian sections  of ${\rm End}(TM)$,   defined by the K\"ahler metrics $\tilde g$ on $(M, J)$  whose Levi-Civita connections are c-projectively equivalent  to $\nabla$. The latter topic has been a subject of intensive study on its own. We will need the following result, which is a combination of Theorems 1 and 3 in   \cite{CMR}:
\begin{thm}\label{thm:CMR}\cite{CMR} Suppose $(M, J, g, \omega)$ is a connected K\"ahler manifold which admits a hamiltonian $2$-form of order $\ell\ge 1$ and satisfies $D(J, g) \ge 3$. Then, in the local form \eqref{k-order-ell} of the K\"ahler structure, all $F_j(t) = F(t)$ for  a (same)  polynomial  $F(t)$ of degree at most $(n+1)$, which is divisible by $p_c(t)$.
\end{thm} 

 \subsection{The hamiltonian $2$-form ansatz for steady K\"ahler-Ricci solitons}
 \smallskip
 The local differential geometry of the K\"ahler structures $(g, \omega)$ of the form \eqref{k-order-ell} is studied in detail in \cite{ACG}. Here we collect the following results from that reference:
 \begin{enumerate}
 \item[$\bullet$] \cite[p.~391]{ACG} Denote by  $\check{h}_a$  a local potential for $\check{\omega}_a$, i.e. $dd^c_B \check{h}_a= \check{\omega}_a$. Then  the smooth functions \[y_r:=-\sum_{a=1}^N(-1)^r \, {\varepsilon_a}\eta_a^{\ell-r}\check{h}_a - \sum_{j=1}^{\ell}\int^{\xi_j}\frac{(-1)^r p_c(t)t^{\ell-r}}{F_j(t)} dt, \]
 where $r=1, \dots, \ell$, are pluriharmonic on $M^{0}$, i.e. they satisfy $dd^c_M y_r=0$.
 \item[$\bullet$] \cite[p.~394]{ACG} If $\check{\kappa}_a$ denotes a local Ricci potential for the K\"ahler form $\check{\rho}_a$ of the K\"ahler metric $\check{\omega}_a$ on $B_a$, i.e. $\check{\rho}_a = dd^c_{B_a} \check{\kappa}_a$, then
 \[ \kappa := \sum_{a=1}^N \check{\kappa}_a -\frac{1}{2}\sum_{j=1}^\ell\log|F_j(\xi_j)|\]
 is a local Ricci potential of $(g, \omega)$, i.e. the Ricci form $\rho$ of $(g, \omega)$ satisfies $\rho=dd^c_M \kappa$.
 \end{enumerate}
 We obtain from the above facts that 
 \begin{lemma}\label{l:KRS-ell} A K\"ahler structure
 $(g, \omega, J)$ defined by \eqref{k-order-ell} gives rise to a steady GKRS provided that each $F_j$ satisfies  \begin{equation}\label{KRS-ell}
 -\frac{1}{2} \log |F_j(\xi_j)| + \frac{1}{2}\int^{\xi_j}\frac{p_c(t)q(t)}{F_j(t)} dt = a\xi_j +b_j,
 % \footnote{\textcolor{blue}{CC:I think the ``+b'' that we had before was confusing to people, since it does not technically have to be the same b for each j. On the other hand, we do technically need the constant term here, since we really solve the differential version of the equation \eqref{ODE-F}, and we cannot guarantee that the RHS is always precisely equal to $a \xi_j$ and not $a \xi_j + const$.  }}
 \end{equation}
 where 
 \[ q(t) = \sum_{r=1}^{\ell} q_r t^{\ell-r}\]
 is a polynomial of degree at most $\ell-1$,  $a, b_j$ are real constants, and, furthermore,  each $(B_a, \check{g}_a, \check{\omega}_a)$ is a K\"ahler-Einstein manifold of scalar curvature $\check{s}_a= -{\rm dim}_{\C}(B_a)q(\eta_a)\varepsilon_a$. In this case, the Killing potential of the soliton vector field $X$ is \[f= 2a(\xi_1 + \cdots + \xi_{\ell})=  2a\sigma_1.\]
 \end{lemma}

 \begin{rem}
Letting  in the above construction 
\[ N=1, \,  \varepsilon_1=-1,  \, \eta_1=0, \, \ell=1,  \, \alpha_1=\alpha \ge 0, \,  \beta_1=+\infty,\]  up to renaming the variables, we obtain the Calabi ansatz in Lemma~\ref{KRS-Calabi}.\end{rem}

  \begin{rem} Although we do not use this below, one can show that Lemma~\ref{l:KRS-ell} describes the general form of a steady GKRS for which the K\"ahler structure $(g, J, \omega)$ is given by \eqref{k-order-ell}  with $\ell \ge 1$  and the soliton Killing potential $f$  is $\T^{\ell}$-invariant.  This can be obtained along  the following lines. As the fibres of $(M^0, g, J) \to B$ are totally geodesic toric manifolds  (see \cite[Prop.~9]{ACG}), $f$ induces on each fibre a $\T^{\ell}$-invariant Killing potential, and thus it has the form
\[ f = \sum_{r=0}^{\ell} \check{f}_r \sigma_r, \]
where $\check{f}_r$ are (pullbacks of) functions on the base.  On the other hand, the expression above of the Ricci potential $\kappa$  of  $(g, J, \omega)$  is separable with respect to the variables $\xi_j$ and $B$. Similarly, one can use the explicit form of the 
complex structure to show that (see e.g. \cite[(4.13), p.~2583]{jubert} for a more general calculation) a $\T^{\ell}$-invariant pluriclosed function is necessarily given by a linear combination of the functions $y_r$ defined above plus the pullback of a pluriclosed function on the base, i.e. it too is separable with respect to the variables $\xi_j$ and $B$.  It then follows that $f$ must also be separable, and therefore is of the form $f=\check{f}_0 + a\sigma_1$ for a constant $a$ and a smooth function $\check{f}_0$ on $B$.  As $\check{f}_0$ is then a Killing potential for the metric \eqref{k-order-ell}, it must be constant as well (this can be checked for instance by using the formulae in  \cite[Prop.~9 and Lemma 4]{ACG}).  \end{rem} 

The condition \eqref{KRS-ell} is equivalent to the (separated) ODE's
\begin{equation}\label{ODE-F}
F_j'(t) + 2aF_j(t) = q(t)p_c(t), \qquad j=1, \ldots, \ell, \end{equation}
whose general solution is (see \eqref{general-solution})
\begin{equation}\label{general-solution-ell}
F_j(t) =e^{-2at}\left(\int_0^t e^{2ax}q(x)p_c(x)dx+ c_j\right), \qquad j=1, \ldots, \ell.
\end{equation}

\section{New examples of toric steady GKRS on $\C^n$}
We discuss here how  the steady GKRS's of the form \eqref{k-order-ell} can be defined on a dense subset $M^{0} \subset \C^{n}$ and smoothly extended to $\C^{n}$.  As mentioned in the introduction, we produce two families corresponding to different choices of parameters in the construction presented in Section \ref{s:hfkg}.  As we will see, these two families turn out to exhibit rather distinct geometric properties. 

\subsection{Generalizing Cao's example: Proof of Theorem~\ref{thm:Cao}}\label{s:type-C^n}
We will first assume that
\[-\infty< \alpha_1<\alpha_2< \cdots < \alpha_{\ell} < \infty, \]
and the variables  $(\xi_1, \ldots, \xi_{\ell})$ belong to the domain
\[\mathring{\rm D} :=(\alpha_1, \alpha_2)\times \cdots \times (\alpha_{\ell-1}, \alpha_{\ell})\times (\alpha_{\ell}, +\infty).\]
For each $\alpha_j, j=1, \ldots, \ell$, we associate a complex projective space $\PP^{d_j}$ of (complex) dimension $d_j$ (which can be also $0$, i.e. we allow for $\PP^{d_j}$ to be a point).  We put the constants $\eta_j:=\alpha_j$ in the construction \eqref{k-order-ell}, so that in our case  
\[ p_c(t)= \prod_{j=1}^{\ell}(t-\alpha_j)^{d_j}. \]
The first inequality in \eqref{inequality} then determines the alternating signs 
\begin{equation}\label{Cao-signs}
\varepsilon_j= (-1)^{\ell-j+1}, \, j=1, \ldots, \ell.\end{equation}

\bigskip
We next construct solutions of \eqref{ODE-F}  for an appropriate choice of the polynomial $q(t)$.
We let
\begin{equation}\label{F}
 F_j(t) = F(t):=e^{-2at}\left(\int_{\alpha_1}^t e^{2ax}q(x)p_c(x)dx \right), \qquad j=1, \ldots, \ell, \end{equation}
so that $F(\alpha_1) = 0$, and impose further~\footnote{In the case $\ell=1, \, d_1=(n-1)$, which corresponds to Cao's construction on $\C^n$ reviewed in Section~\ref{s:calabi-type}, this extra-assumption is vacuous.} 
\[F(\alpha_j)=0, \qquad j=2, \ldots, \ell. \]
The latter condition places $(\ell-1)$ linear constraints on the coefficients of the polynomial $q(t)$. As $q(t)$  must be of degree $\le (\ell-1)$,  this should determine $q(t)$ up to an overall multiplicative constant $c$.  To establish this rigorously, observe that for any polynomial $Q(t)$ of degree $\le (n-1)$,   the general solution of the ODE $F'(t) + 2a F(t) = Q(t)$ (with $a\neq 0$) is of the form $F(t) = P(t) + c e^{-2at},$ where $P(t)$ is a polynomial of degree $\le (n-1)$. In our case, $Q(t)=p_c(t) q(t)$,   and then the conditions that $F(\alpha_j)=0, \, j=1, \ldots, \ell$ actually yield that each $\alpha_j$ must be  a zero of $F(t)$ of order at least $(d_j+1)$. For fixed $c, \alpha_1, \ldots, \alpha_{\ell}$, the latter conditions on $F(t)$ prescribe the values of $P(t)$ up to order $d_j$ at each $\alpha_j$. By the Lagrange--Sylvester interpolation formula,  this determines  $P(t)=P_{a,c}(t)$ with $c$ being a scale constant. Given $P_{a,c}(t)$, the corresponding $q(t)=q_{a,c}(t)$ is then defined from the equation \eqref{ODE-F} for $F(t)=P_{a,c}(t) + c e^{-2at}$, i.e.
\[ q(t) = \frac{F'(t) + 2a F(t)}{p_c(t)}=\frac{P_{a,c}'(t) + 2a P_{a, c}(t)}{p_c(t)}. \]
The divisibility of the polynomials at the RHS is insured by the fact that  each $\alpha_j$ is a zero of order at least $d_j$ of $F'(t) + 2a F(t)$.

\smallskip
In order to satisfy the second inequality in \eqref{inequality},  we let 
\[\Theta(t) := \frac{F(t)}{p_c(t)}. \] The inequality then is 
\begin{equation}\label{positivity-Theta}
(-1)^{\ell-i}\Theta(t) >0 \, \, \mathrm{on} \, \, (\alpha_i, \alpha_{i+1}), \, \,  i=1, \ldots, \ell-1, \qquad \Theta(t)>0 \, \, \mathrm{on} \, \, (\alpha_\ell, \infty).\end{equation}
As $F(t)$ is an analytic function on $\R$, the ODE \eqref{ODE-F} yields that $\Theta(t)$ is a smooth function satisfying
\begin{equation}\label{ODE-Theta}
\Theta'(t) + \left( 2a+ \sum_{j=1}^\ell \frac{d_j}{(t-\alpha_j)}\right)\Theta(t)= q(t).\end{equation}
Taking the limit $t\to \alpha_i$, \eqref{ODE-Theta} shows that 
\begin{equation}\label{boundary-Theta}
\Theta(\alpha_i)=0, \qquad (d_j+1)\Theta'(\alpha_i)= q(\alpha_i), \qquad i=1, \ldots, \ell.\end{equation} 
As $p_c(t)$ has a constant sign on each interval $(\alpha_i, \alpha_{i+1})$ and we have $\int_{\alpha_i}^{\alpha_{i+1}} e^{2ax}q(x)p_c(x)dx =0$, it follows that $q(t)$ must have a root in each interval  $(\alpha_i, \alpha_{i+1}), i=1, \ldots, \ell-1$.  As $q(t)$ has degree at most $(\ell-1)$, it must have a simple root at each of the intervals (and the degree of $q(t)$ must be $(\ell-1)$ if $a\neq 0$). It thus follows that $q(\alpha_i)= (d_j+1)\Theta'(\alpha_i)$ changes its sign as $i=1, \ldots, \ell$.  A little calculus, using that $q(t)$ has a unique simple root in $(\alpha_i, \alpha_{i+1})$ and \eqref{ODE-Theta}, shows further that  $\Theta(t)$ cannot have a zero in $(\alpha_i, \alpha_{i+1})$ and in $(\alpha_{\ell}, +\infty)$.
Changing the sign of $q(t)$  if necessary (and thus of $F(t)$ and $\Theta(t)$), we can assume that both inequalities in \eqref{inequality} are 
satisfied,  i.e. (see \eqref{Cao-signs}) we have on $\mathring{\rm {D}}$
\begin{equation}\label{Cao-signs-final}
\varepsilon_j p_{nc}(\alpha_j)>0, \qquad -\varepsilon_j q(\alpha_j) >0, \qquad j=1, \ldots, \ell.
\end{equation} 

\bigskip
We now construct a principal $\T^{\ell}$-bundle $P\to \prod_{j=1}^{\ell} \PP^{d_j}$ in order to apply the ansatz \eqref{k-order-ell} on
\[ M^{0}:= (\alpha_1, \alpha_2)\times \cdots \times (\alpha_{\ell -1}, \alpha_{\ell})\times(\alpha_{\ell}, +\infty) \times P \to B:=\prod_{j=1}^{\ell} \PP^{d_j}.\]
We let $\check{\omega}_j$ be a Fubini--Study metric on $\PP^{d_j}$ of scalar curvature (see Lemma~\ref{l:KRS-ell} and \eqref{boundary-Theta})
\[ s_j = -\varepsilon_j d_j(d_j+1) \Theta'(\alpha_j) = - \varepsilon_j d_j q(\alpha_j), \qquad j=1, \ldots, \ell.\]
The positivity of $s_j$ is then consistent with \eqref{inequality} (see \eqref{positivity-Theta} and \eqref{boundary-Theta}). Notice that
\[ \check{\omega}_j = -\left(\frac{2}{\varepsilon_j \Theta'(\alpha_j)}\right) \check{\omega}_j^0= -\left(\frac{2(d_j+1)}{\varepsilon_j q(\alpha_j)}\right)\check{\omega}_j^0, \]
where $\check{\omega}_j^0$ is the Fubini--Study metric on $\PP^{d_j}$of constant scalar curvature $2d_j(d_j+1)$, so that $[\check{\omega}_j^0]$ is the primitive generator of $H^2(\PP^{d_j}, \Z)$. We let
\begin{equation}\label{vj-def}
 v_j := \left(\frac{2(d_j+1)}{q(\alpha_j)}\right)\Big(\alpha_j^{\ell-1}, \ldots, (-1)^{r-1}\alpha_j^{\ell-r}, \ldots, (-1)^{\ell-1} \Big)  \in \R^n, % \qquad j=1, \ldots, \ell.
 \end{equation}
where $j=1, \ldots, \ell.$ As $\alpha_j$ are pairwise distinct, $\{v_j\}$ is a basis of $\R^\ell$ and we denote by $\Gamma_v\cong \Z^{\ell}$ the lattice spanned by this basis; we then consider $\T^{\ell}:= \R^\ell/2\pi\Gamma_v$ the corresponding compact torus. We can rewrite the condition \eqref{theta} as
\[ d\theta = \sum_{j=1}^{\ell} \check{\omega}_j^0 \otimes v_j, \]
showing that there exists a principal $\T^{\ell}$-bundle $P$ over $B=\prod_{j=1}^{\ell} \PP^{d_j}$ with connection $1$-form $\theta$ as above. 

\begin{conclusion}\label{construction}
By Lemma~\ref{l:KRS-ell}, the K\"ahler metric $(g, J, \omega)$ defined by \eqref{k-order-ell}  and the data $q(t), a, \alpha_1<\cdots < \alpha_{\ell}$ and $\eta_j = \alpha_j$ 
 gives rise to a steady GKRS metric on $M^{0}$.
\end{conclusion}

To understand further the global properties of the K\"ahler metric thus defined, we notice that the $\T^{\ell}$-momenta 
\[ \sigma =\Big(\sigma_1(\xi_1,\ldots, \xi_{\ell}), \ldots, \sigma_{\ell}(\xi_1, \ldots, \xi_{\ell})\Big)\]
are given by the elementary symmetric functions $\sigma_i$ of the variables $(\xi_1, \ldots, \xi_{\ell})$, and thus define a diffeomorphism
\[ \sigma : (\alpha_1, \alpha_2)\times \cdots \times (\alpha_{\ell -1}, \alpha_{\ell})\times(\alpha_{\ell}, +\infty) \cong \mathring{\Pol}_{\ell} \subset \R^{\ell},\]
where 
%\[ \mathring{\Pol}_{\ell} : = \left\{ (\sigma_1, \ldots, \sigma_{\ell})  \in \R^{\ell}:  L_j(\sigma):= \left(\langle v_j, \sigma \rangle  - \frac{2\alpha_j^{\ell}(d_j+1)}{q(\alpha_j)}\right)> 0, \, \, j=1, \ldots, \ell \right\}. \]
\[ \mathring{\Pol}_{\ell} : = \left\{ (\sigma_1, \ldots, \sigma_{\ell})  \in \R^{\ell} \, \left| \, \begin{array}{l}  L_j(\sigma):= \left(\langle v_j, \sigma \rangle  - \frac{2\alpha_j^{\ell}(d_j+1)}{q(\alpha_j)}\right)> 0,  \\  j=1, \ldots, \ell \end{array} \right. \right\}. \]
By the definition \eqref{vj-def} of $v_j$, the affine-linear functions $L_j(\sigma)$ can be expressed as
\begin{equation}\label{fiber-label}
L_j(\sigma)= -\frac{2(d_j+1)}{q(\alpha_j)}\left(\sum_{r=0}^{\ell}(-1)^r\alpha_j^{\ell-r}\sigma_r\right)=-\frac{2(d_j+1)}{q(\alpha_j)}p_{nc}(\alpha_j), % \qquad j=1, \ldots, \ell.
\end{equation}
where $j=1, \ldots, \ell.$ These define a labelled cone  $\left({\Pol}_{\ell}, L_j\right)$ in $\R^{\ell}$  which, with respect to the lattice generated by $\{v_1, \ldots, v_{\ell}\}$,  is the Delzant polyhedron of the flat standard symplectic structure on $\R^{2\ell}$. Indeed, by \eqref{Cao-signs-final},  $L_j(\sigma)=-\frac{2(d_j+1)}{q(\alpha_j)}p_{nc}(\alpha_j)>0$ on $\mathring{\Pol}_\ell$.%\footnote{\textcolor{blue}{CC:Do we need to explain later what happens in the Taub-NUT case? Something different also happens in that case when $a = 0$, since the degree of $q$ drops by 1. {\color{blue} VA: I don't see a difference in these cases. I corrected the definition of $\varepsilon_j$ in the Taub-NUT case,  and I fixed a sign error in $\varepsilon_j q(\alpha_j)$ in the Cao case. The right inequalities come from $\varepsilon_{\ell}p_{nc}(\alpha_{\ell}) >0$ against $\Theta_{\ell}'(\alpha_{\ell})>0$. The  first inequality is a part of the ansatz whereas the latter inequality holds for completions with infinite end $(\alpha_{\ell}, + \infty)$.} } }

 It thus follows that $M^{0} \cong (\C^{\times})^{\ell}\times_{\T^{\ell}} P \to B$ is,  in fact,  diffeomorphic  (but not in general biholomorphic for the induced complex structure via $(g, \omega)$) to a principal $(\C^{\times})^{\ell}$-bundle over $B=\prod_{j=1}^{\ell}\PP^{d_j}$,  corresponding  to  the split vector bundle
\[M^{0} \subset  \hat M:=\left(\bigoplus_{j=1}^{\ell}\cO_{\PP^{d_j}}(-1)\right) \to \prod_{j=1}^{\ell}\PP^{d_j} \]
thus generalizing the setting  in the previous section; notice that under the blow-down map $B: \hat M \to \C^n$, we have that $M^0 \subset \C^n\cong \R^{2n}$. 

Noting that  the base $B= \prod_{j=1}^{\ell}(\PP^{d_j}, \check{\omega}_j^0)$ is  itself a smooth toric manifold,  with corresponding Delzant polytope
\[\Pol_B = \prod_{j=1}^{\ell} \Pol_{\PP^{d_j}}\]
where each $\Pol_{\PP^{d_j}} \subset \R^{d_j}$ is  a the standard simplex 
\[\Pol_{\PP^{d_j}}=\left\{(\check{\mu}^j_{1}, \ldots, \check{\mu}^{j}_{d_j}) \, : \, \check{\mu}^j_k \ge 0, \, k=1, \ldots, d_j, \,  (1-\sum_{k=1}^{d_j} \check{\mu}^{j}_k)\ge 0\}\right\},\]
we are going to describe $(M^0, \omega)$ as a toric $2n$-dimensional manifold. To this end, denote by 
\begin{equation}\label{check-label}
 \check{L}^{j}_k(\check{\mu}^{j}) := \check{\mu}^j_{k},  \, \, k=1, \ldots d_j, \qquad \check{L}^{j}_0(\check{\mu}^j) : = (1-\sum_{k=1}^{d_j} \check{\mu}^{j}_k)\end{equation}
 the  affine-linear functions (which we refer to as \emph{labels} of $\Pol_{\PP^{d_j}}$)  defining the Delzant polytope of each factor $\PP^{d_j}$ of the base. The toric action of $\check{\T}^{n-\ell} = \prod_{j=1}^{\ell} \check{\T}^{d_j}$ on the base lifts to define a $\T^{n}$-action on the total space $\hat M=\left(\bigoplus_{j=1}^{\ell}\cO_{\PP^{d_j}}(-1)\right) \to \prod_{j=1}^{\ell}\PP^{d_j}$ under which $(M^{0}, \omega, J, g)$ is invariant. As a matter of fact,  $(\omega, J, g)$ is obtained by the \emph{generalized Calabi construction} (see e.g. \cite{ACGT1}) and we know 
from \cite[Lemma~5]{ACGT1} that
\begin{equation}\label{mu}
\mu= (\sigma, L_1(\sigma)\check{\mu}^{1}, \ldots, L_{\ell}(\sigma)\check{\mu}^{\ell})
\end{equation}  are $\omega$-momenta for the induced $\T^n$-action on $M^0$, where we recall $L_j(\sigma)$ are the labels of $\mathring{\Pol}_{\ell}$ introduced in \eqref{fiber-label}. It follows that \eqref{mu} defines a diffeomorphism
\[ \mu : \mathring{\Pol}_{\ell}\times \mathring{\Pol}_{\PP^{d_1}} \times \cdots \times \mathring{\Pol}_{\PP^{d_\ell}} \cong \mathring{\Pol}_n, \]
where 
\[ \mathring{\Pol}_n := \left\{\mu \in \R^n \, \,  | \,  \,  L_j(\sigma)\check{L}^j_k(\check{\mu}^j) >0, \,   \qquad j=1, \ldots \ell, \, k=0, \ldots d_j\right\}, \]
and  $\mu(M^0)\subset \Pol_n$.  Notice that $L_j(\sigma)\check{L}^j_k(\check{\mu}^j)$ are affine-linear in $\mu$ (see \eqref{mu}) and define a labeled cone $\left(\Pol_n, L_j(\sigma)\check{L}^j_k(\check{\mu}^j)\right)$ in $\R^n$, corresponding to the flat $(\R^{2n}, \omega_0)$ (seen as a toric manifold via the Delzant description).

\smallskip
We now prove the following
\begin{lemma}\label{irreducible} The steady GKRS $(g, \omega)$ defined on $M^0$ as above extends to a smooth steady GKRS metric on $\R^{2n}=\prod_{j=1}^{\ell} \R^{2(d_j+1)}$, invariant under the standard $\T^n$-action, and compatible with the standard symplectic form. Moreover, the soliton vector field is
\[   2a\left(\sum_{j=1}^{\ell} \frac{q(\alpha_j)}{2(d_j +1)\prod_{k\neq j}(\alpha_j-\alpha_k)}X_j\right), \]
where $X_j$ is the vector field on $\R^{2(d_j+1)}\cong \C^{d_j+1}$ whose flow is multiplication with $e^{i2\pi t}$;  $(g, \omega)$ is the flat K\"ahler metric on $\R^{2n}\cong \C^n$ for $a=0$ and is an irreducible GKRS for $a\neq 0$. 
\end{lemma}

\begin{proof} On $M^0= \mu^{-1}(\mathring{\Pol}_n) \cong \mathring{\Pol}_n\times \T^n\cong (\C^{\times})^n$,  the GKRS $(g, \omega)$ admits a description in terms of a symplectic potential $U(\mu)$. From  \cite[Lemma~6]{ACGL} 
\begin{equation}\label{symplectic-potential}
U(\mu):= u(\sigma) + \sum_{j=1}^{\ell}L_j(\sigma)\check{u}^j(\check{\mu}^j)\end{equation}
where  $u(\sigma)$ is a fibre-wise symplectic potential given by (see \cite[Prop.~11]{ACG})
\[ u(\sigma)= -\sum_{j=1}^{\ell}\int^{\xi_j}\frac{p_{nc}(t)}{\Theta_j(t)} dt, \qquad \Theta_j(t) := \frac{F_j(t)}{p_c(t)},  \]
and  
\[\check{u}^j(\check{\mu}^j)= \frac{1}{2}\left(\sum_{k=0}^{d_j} \check{L}^j_k(\check{\mu}^j) \log \check{L}^j_k(\check{\mu}^j)\right)\] are the symplectic potentials of the Fubini--Study metrics $\check{\omega}^0_j$ on $\PP^{d_j}$ (see \cite{guillemin}). Similarly to the proof of \cite[Prop.~14]{AR}, the formula \eqref{symplectic-potential} allows us to show the smooth extension of the metric to $\R^{2n}$, viewed as a toric manifold with Delzant image $\Pol_n$. Indeed, using that $dL_j(\sigma)=v_j =\frac{2}{\Theta'(\alpha_j)}\Big(\alpha_j^{\ell-1}, \ldots, (-1)^{r-1}\alpha_j^{\ell-r}, \ldots, (-1)^{\ell-1} \Big)$, it follows from \cite[Prop.~9]{ACGT} that at the boundary of $\Pol_{\ell}$,  ${\bf H}_{\sigma}:={\rm Hess}_{\sigma}(u)^{-1}$ satisfies 
\[{\bf H}(\sigma)(dL_j, \cdot)=0, \qquad d{\bf H}_{\sigma}(dL_j, dL_j)=2dL_j, \qquad j=1, \ldots, \ell. \]
This in turn yields (see \cite[Thm.2]{Abreu1}) that
\[u(\sigma) = \frac{1}{2}\sum_{j=1}^{\ell} L_j(\sigma) \log L_j(\sigma) + \textrm{smooth}.\]
We then estimate
%\[
%\begin{split}
%U(\mu) &= \sum_{j=1}^{\ell} \left( \frac{1}{2}L_j(\sigma) \log L_j(\sigma) + L_j(\sigma)\check{u}^j(\check{\mu}^j)\right) + \mathrm{smooth} \\
%&=\frac{1}{2}\sum_{j=1}^{\ell}\left(L_j(\sigma) \log L_j(\sigma) + \sum_{k=0}^{d_j} L_j(\sigma)\check{L}^j_k(\check{\mu}^j)\log\check{L}^j_k(\check{\mu}^j) \right) + \mathrm{smooth} \\
%&= \frac{1}{2}\left(\sum_{j=1}^{\ell}\sum_{k=0}^{d_j} \left(L_j(\sigma)\check{L}^j_k(\check{\mu}^j)\right)\log \left(L_j(\sigma)\check{L}^j_k(\check{\mu}^j)\right)\right)  + \mathrm{smooth},
%\end{split}\]
\[
\begin{split}
U(\mu)\,\, &\sim\, \sum_{j=1}^{\ell} \left( \frac{1}{2}L_j(\sigma) \log L_j(\sigma) + L_j(\sigma)\check{u}^j(\check{\mu}^j)\right) \\
&\sim\, \frac{1}{2}\sum_{j=1}^{\ell}\left(L_j(\sigma) \log L_j(\sigma) + \sum_{k=0}^{d_j} L_j(\sigma)\check{L}^j_k(\check{\mu}^j)\log\check{L}^j_k(\check{\mu}^j) \right)  \\
&\sim\,  \frac{1}{2}\left(\sum_{j=1}^{\ell}\sum_{k=0}^{d_j} \left(L_j(\sigma)\check{L}^j_k(\check{\mu}^j)\right)\log \left(L_j(\sigma)\check{L}^j_k(\check{\mu}^j)\right)\right) ,
\end{split}\]
where $\sim$ denotes equality up to the addition of a function extending smoothly over the boundary of $\Pol_n$, and to get the last line we have used that $\sum_{k=0}^{d_j}\check{L}^j_k(\check{\mu}^j)=1$.

Similarly,  
\begin{equation}\label{det}
 \det\left({\rm Hess}(U)\right)  \left(\prod_{j=1}^{\ell} \prod_{k=0}^{d_j} L_j(\sigma)\check{L}_k^j(\check{\mu}^j)\right)\end{equation}
extends as a smooth positive function over the boundary of $\Pol_n$.  We conclude by \cite[Thm.2]{Abreu1} that  $(g, \omega)$ extends to a smooth $\T^n$-invariant K\"ahler metric on $\R^{2n}$.

In order to recast the soliton vector field $X$ of $(g, \omega)$ (which corresponds to the Killing potential $a(\xi_1 + \cdots + \xi_\ell)=a\sigma_1$ in the description \eqref{k-order-ell} of the metric), we consider $\R^{2n}\cong \C^n = \prod_{j=1}^{\ell} \C^{d_j+1}$ as a smooth symplectic toric manifold with labeled Delzant polyhedron given by the cone $\Pol_n \in \R^n$ with labels 
$L_j(\sigma)\check{L}^j_k(\check{\mu}^j_k), \, j=1, \ldots, \ell, \, k=0, \ldots, d_j$ (which are  affine-linear functions in $\mu$, see \eqref{mu}). This  identifies $\omega$ with the standard symplectic structure on $\R^{2n}$.
With respect to these labels, we have (see \eqref{check-label})
\[L_j(\sigma) = \sum_{k=0}^{d_j}L_j(\sigma)\check{L}^j_k(\check{\mu}^j_k)\]
and, using \eqref{fiber-label} and the Vandermonde identity (see \cite[App.~B3]{ACG}
\begin{equation}\label{Vandermonde}
\sum_{j=1}^\ell  \frac{\alpha_j^{\ell-r}}{\Delta(\alpha_j)} = \delta_{r1}, \qquad r=1, \ldots, \ell, \end{equation}
(where we recall, $\Delta(\alpha_j):= \prod_{k\neq j}(\alpha_j-\alpha_k)$), we find 
%\begin{equation*}
%\begin{split} &\sum_{j=1}^{\ell} \left(\frac{q(\alpha_j)}{2(d_j+1)\Delta(\alpha_j)}\right)\left(\sum_{k=0}^{d_j} L_j(\sigma)\check{L}^j_k(\check{\mu}^j_k)\right) =\sum_{j=1}^{\ell} \left(\frac{q(\alpha_j)}{2(d_j+1)\Delta(\alpha_j)}\right)L_j(\sigma) \\
%& = -\sum_{j=1}^{\ell}\left(\sum_{r=0}^{\ell} (-1)^r \frac{\alpha_j^{\ell-r}}{\Delta(\alpha_j)} \sigma_r\right)= \sigma_1.
%\end{split}
%\end{equation*}
%\begin{equation*}
%\textcolor{red}{\begin{split} &\sum_{j=1}^{\ell} \left(\frac{q(\alpha_j)}{2(d_j+1)\Delta(\alpha_j)}\right)\left(\sum_{k=0}^{d_j} L_j(\sigma)\check{L}^j_k(\check{\mu}^j_k)\right) \\
%& \hspace{2in} =\sum_{j=1}^{\ell} \left(\frac{q(\alpha_j)}{2(d_j+1)\Delta(\alpha_j)}\right)L_j(\sigma) \\
%& \hspace{2in} = -\sum_{j=1}^{\ell}\left(\sum_{r=0}^{\ell} (-1)^r \frac{\alpha_j^{\ell-r}}{\Delta(\alpha_j)} \sigma_r\right)= \sigma_1.
%\end{split}}
%\end{equation*}
\begin{equation*}
\begin{split} &\sum_{j=1}^{\ell} \left(\frac{q(\alpha_j)}{2(d_j+1)\Delta(\alpha_j)}\right)\left(\sum_{k=0}^{d_j} L_j(\sigma)\check{L}^j_k(\check{\mu}^j_k)\right) \\
& =\sum_{j=1}^{\ell} \left(\frac{q(\alpha_j)}{2(d_j+1)\Delta(\alpha_j)}\right)L_j(\sigma)  = -\sum_{j=1}^{\ell}\left(\sum_{r=0}^{\ell} (-1)^r \frac{\alpha_j^{\ell-r}}{\Delta(\alpha_j)} \sigma_r\right)= \sigma_1.
\end{split}
\end{equation*}
This yields the expression of the soliton vector field in the statement of the lemma.

When $a=0$, the corresponding function $F$  (see \eqref{F}) is a degree $n$ polynomial which satisfies $F'(t) = q(t)p_c(t)$. As $F(\alpha_j)=0$ and $p_c(t)=\prod_{j=1}^{\ell} (t-\alpha_j)^{d_j}$, we see that $F(t)$ is a multiple of $\prod_{j=1}^{\ell}(t-\alpha_j)^{d_j+1}$. The  resulting K\"ahler structure is flat according to \cite[Prop.~17]{ACG}.

Finally, as the metric $(g, \omega)$ on $M^0$ admits a non-trivial hamiltonian $2$-form of order $\ell \ge 1$  (see \cite[Thm.1]{ACG}) and $F_j(t)=F(t)$ is not a polynomial when $a\neq 0$ (by the argument invoked above),  Theorem~\ref{thm:CMR} tells us that the space of all hamiltonian $2$-forms on $(M^0, g, \omega)$ is of real dimension $2$. In particular, $(M^0, g, \omega)$ cannot be a (local) product of K\"ahler manifolds as if it were, the metric would admit additional one parameter family of non-trivial hamiltonian $2$-forms of order $0$. \end{proof}

\begin{rem}\label{Cao-vf}
Since $q(\alpha_j)$ have alternating signs,  one sees  from Lemma~\ref{irreducible}  that the soliton vector field is of the form $a\left(\sum_{j=1}^\ell \lambda_j X_j\right)$, for $\lambda_j > 0$.   Lemma \ref{frlemma} tells us that we should impose $a\ge 0$ if we want the corresponding complex structure $J$ be isomorphic to $J_0$.
\end{rem}

\begin{lemma}\label{l:extension} Given integers $\ell\ge 1, d_1, \ldots, d_{\ell}\ge 0$ with $n=\ell + \sum_{j=1}^{\ell}d_j$, the above construction defines two $(\ell-1)$-dimensional continuous families of non-isometric up to scale,   $\T^n$-invariant non-flat steady GKRS on $\R^{2n}$, compatible with the standard symplectic structure.
\end{lemma}
\begin{proof}  Given $\ell\ge 1, d_1, \ldots, d_{\ell}\ge 0$ as in the lemma, the GKRS  $(g, \omega)$ on $M^0$ are constructed from the data of real numbers
\[ a \in \R, \qquad \alpha_1<\alpha_2<\cdots <\alpha_{\ell}, \qquad \lambda>0,\]
where $\lambda$ is a free multiplicative constant before $q(t)$ (and hence $F(t)$) as determined in the construction.

\bigskip
Let us first suppose $\ell \ge 2$. Any K\"ahler metric $(g, \omega)$  of the form \eqref{k-order-ell} admits a hamiltonian $(1,1)$-form $\phi$ of order $\ell$, given explicitly on $M^0$ by \eqref{eq:hamiltonian-2-form}; furthermore, $(\xi_1, \ldots, \xi_{\ell}, \eta_1, \ldots, \eta_N)$ are the eigenvalues of $\phi\circ \omega^{-1}$. As we can replace the hamiltonian $2$-form $\phi$ by $p\phi + q \omega$, we can change the data $(\xi_1, \ldots, \xi_{\ell}, \eta_1, \ldots, \eta_N)$  by a simultaneous affine transformation without changing the resulting K\"ahler metric $(g, \omega)$. For our specific construction (where $\eta_a=\alpha_j$), this  allows us to fix $\alpha_1=0, \alpha_2=1$. Furthermore, changing $F(t)$ by $\lambda F(t)$ results in a scale by a factor $\frac{1}{\lambda}$ of the labels $L_j(\sigma)\check{L}^j_k(\check{\mu}^j)$ of the Delzant cone $\Pol_n$ and to the rescaled K\"ahler metric $(\lambda g, \lambda \omega)$ (see e.g. formula \eqref{symplectic-potential} for the symplectic potential of the metric). We can then normalize this constant too by assuming $q(\alpha_1)=(-1)^{\ell}$. We are thus left with $(\ell-1)$ effective parameters $(a, \alpha_3, \ldots, \alpha_{\ell})$. 

If $a=0$, we already observed in Lemma~\ref{irreducible} that the resulting K\"ahler metric is flat, so we get isometric metrics for all values of $(\alpha_3, \ldots, \alpha_{\ell})$.

Suppose now $a\neq 0$. We claim that two K\"ahler metrics  $(g, \omega)$ and $(\tilde g, \tilde \omega)$ corresponding to data 
\[ (a, 0, 1, \alpha_3, \ldots, \alpha_{\ell}) \neq (\tilde a, 0, 1, \tilde \alpha_3, \ldots, \tilde \alpha_{\ell}), \qquad a\neq 0 \]
are not even locally isometric up to scale. Suppose for contradiction that $(g, \omega)$ and $(\tilde g, \tilde \omega)$ where isometric up to scale over some open subsets.  It  follows from Lemma~\ref{irreducible} that $\tilde a \neq 0$, and thus the space of hamiltonian $2$-forms associated to each $(g, \omega)$ and $(\tilde g, \tilde \omega)$ is $2$-dimensional by  Theorem~\ref{thm:CMR}. Identifying $(\tilde g, \tilde \omega)$ with $(\lambda g, \lambda  \omega)$ under the local isometry, it follows that the corresponding hamiltonian $2$-form $\tilde \phi$ is identified with $p\phi + q\omega, \, p\neq 0$.  Thus the eigenvectors of $\tilde \phi \circ \tilde \omega^{-1}$ and $\phi \circ \omega^{-1}$ are related by simultaneous affine transformation. As $\alpha_1=0=\tilde \alpha_1$ and $\alpha_2 =1=\tilde \alpha_2$ by our normalization, we conclude that $\alpha_3=\tilde \alpha_3, \ldots, \alpha_{\ell}= \tilde \alpha_{\ell}$ (and thus $\tilde \phi$ is sent to $\lambda \phi$ under the local isometry).  As $\sigma_1 = \tr \left(\phi \circ \omega^{-1}\right)$ and $\tilde \sigma_1 =\tr \left(\tilde \phi \circ \tilde \omega^{-1}\right)$ are identified under the isometry, and as (by the soliton condition)  the corresponding Ricci tensors are $\rho= a (dd^c \sigma_1)$ and $\tilde \rho = \tilde a (dd^c \tilde \sigma_1)$, we conclude $a=\tilde a$, a contradiction. 

We thus obtain two families respectively parametrized by 
\[\{ (a, \alpha_2, \ldots, \alpha_{\ell})\, | \, a<0, \, 1<\alpha_2<\cdots<\alpha_{\ell}\} \]
and
\[\{ (a, \alpha_2, \ldots, \alpha_{\ell})\, | \, a>0, \, 1<\alpha_2<\cdots<\alpha_{\ell}\}, \]
each obtained by deforming the flat metric at $a=0$.

\bigskip
Now consider the case $\ell=1$. For $a\ge 0$ the metrics in this case are Cao solitons on $\C^n$ and have a ${\rm U}(n)$ symmetry by Remark~\ref{symmetry} below. The uniqueness then follows by \cite{Cao}.  To cover all values of $a$ we argue as in the case $\ell \ge 2$. First, keeping the metric $(g, \omega)$ fixed and changing the corresponding hamiltonian $2$-form, we can set $\alpha_1=0$ as in Section~\ref{s:Cao}. This stills allows us to rescale $\phi$ by a factor $\lambda\neq 0$, which leads to changing  in the ansatz $\xi$ with $\lambda\xi$,  $\Theta_{0,a}(\xi)$ with $\frac{1}{\lambda^2}\Theta_{0,a}(\lambda \xi)$, and rescaling the Fubini--Study metric $\check{\omega}_{\rm FS}$ by a factor of $\frac{1}{\lambda}$ (compare with Remark~\ref{r:homothety}). Thus, we can further fix $q(t)$ (or equivalently the positive constant $\lambda$) in this case without changing the K\"ahler metric $(g, \omega)$. The only constant left is $a$ which can be made equal to $-1, 0, 1$ by  a further scaling factor  of $(g, \omega)$. We then get two non-flat steady GKRS metrics on $\R^{2n},$ one of which ($a=1$) corresponds to Cao's complete GKRS on $\C^n$ and the other $(a=-1)$ is incomplete.
\end{proof}

\begin{lemma}\label{l:completion}  If $a\ge 0$, the steady GKRS solitons on $\R^{2n}$ defined as above are complete.
\end{lemma}
\begin{proof}  When $a=0$ the metric $g$ is the flat metric on $\R^{2n}$. We thus assume $a>0$. By the Hopf--Rinow theorem,  $(\R^{2n}, g)$ is complete if and only if bounded subsets  of $\R^{2n}$ with respect to the distance $d^g$ are bounded in the usual topology.   We are going to show that the latter condition holds,  provided that $F(t) = O(t^{n-1})$ where $F(t)$ is the function \eqref{F}. (Notice that in our case, $F(t) = P(t) + ce^{-2at}$ for a polynomial  $P(t)$ of degree at most $(n-1)$.)

The variables $(\xi_1, \ldots, \xi_{\ell})$ in the description \eqref{k-order-ell} correspond to the (non-constant) eigenvalues of a hamiltonian $(1,1)$-form of order  $\ell$ with respect to $(g, \omega)$, see \cite[Thm.1]{ACG}. This form is defined on $M^0$ but also extends to $\C^n$ because the K\"ahler structure $(g, \omega)$ extends and $\C^{n}\setminus M^0$ is of real co-dimension at most $2$.  (We used here that hamiltonian $2$-forms correspond to parallel sections of an extended tractor bundle, see \cite{ACG}.) Thus, $(\xi_1, \ldots, \xi_{\ell})$ extend to smooth functions on $\C^n$. As we have used the standard Delzant cone $\Pol_n$ of $\C^n$ to extend $\omega$, the momentum map $\mu$ is proper. From \eqref{mu}, it follows that $\sigma:=(\sigma_1, \ldots, \sigma_{\ell})$ is proper and hence also $(\xi_1, \ldots, \xi_{\ell})$. As $(\xi_1, \xi_2, \ldots, \xi_{\ell-1})$ are bounded, to establish the completeness of $g$ it is enough to prove that on any $d^g$-bounded set, $\xi_{\ell}$ is bounded.

We take a curve $\gamma(t), t\in [0, \infty)$ in $\C^n$ issued from a point $p_0$  and write $\xi_i(t)= \xi_i(\gamma(t))$ for the corresponding functions. We assume (without loss for our argument) that $\xi_{\ell}(t)$ goes to $+\infty$ and take $s>>0$ so that $\xi_{\ell}(s)>0$.  From the form \eqref{k-order-ell} of $g$ and using $F(t)=O(t^{n-1})$, we obtain for the length of $\gamma(t), t\in [0, s]$:
\[L^g(\gamma(t)) \ge \int_0^{s} \sqrt{\frac{p_c(\xi_{\ell}(t))\prod_{j\neq 1}(\xi_{\ell}(t)-\alpha_{j})}{F(\xi_{\ell}(t))}} |\dot \xi_{\ell}(t)|dt \ge C \xi_{\ell}(s) + C',  \]
where $C>0, C'$ are uniform constants.
It thus follows that for $p$ outside a compact set, 
\begin{equation}\label{distance}
d^g(p_0, p) \ge C\xi_{\ell}(p) + C', 
\end{equation}
which yields the result.
\end{proof}

\begin{lemma}\label{l:biholomorphic} If $a\ge 0$, the complex structure of the GKRS on $\R^{2n}$ defined as above is equivariantly biholomorphic to the standard complex structure on $\C^n$. In particular, we have an $(\ell-1)$-dimensional family of complete non-flat steady GKRS on $\C^n$. 
\end{lemma}
\begin{proof} Once again, the case $a=0$ is clear as the K\"ahler structure is flat. We assume $a>0$.  By virtue of Proposition~\ref{biholomorphic} we want to show that if $K_r, \, r=1, \ldots, \ell$ are the Killing vector fields of $(g, \omega)$ with Killing potentials $\sigma_r$ and 
\begin{equation}\label{lifted}
X^j_k= \left(\check{X}^j_k\right)^H + \check{\mu}^j_k dL_j
\end{equation}
are the Killing vector fields coming from the lifted torus action on $\PP^{d_j}$, then $JK_r$ and $JX^j_k$ are complete vector fields on $M$.

To this end,  we shall use Lemma~\ref{criterion}. Let us consider first $JK_r$. By \eqref{k-order-ell}, we compute
\[
\begin{split}
g(JK_r, JK_r) &= g(K_r, K_r) = \sum_{j=1}^{\ell}\left(\frac{F(\xi_j)}{p_c(\xi_j)\Delta(\xi_j)}\sigma^2_{r-1}(\hat \xi_j)\right) \\
                      &=\sum_{j=1}^{\ell}\left(\frac{\prod_{k=1}^{\ell}(\xi_j - \alpha_k)^{d_j+1} G(\xi_j)}{p_c(\xi_j)\Delta(\xi_j)}\sigma^2_{r-1}(\hat \xi_j)\right) \\
                      &= \sum_{j=1}^{\ell}\left(\frac{\prod_{k=1}^{\ell}(\xi_j - \alpha_k)G(\xi_j)}{\Delta(\xi_j)}\sigma^2_{r-1}(\hat \xi_j)\right),
\end{split}\]
where we have used that $F(t)$ is a real-analytic function with zeros $\alpha_j$ of multiplicity $d_j+1$, so that $G(t):= \frac{F(t)}{\prod_{k=1}^\ell(t-\alpha_k)^{d_j+1}}$ is a smooth function which is $O(\frac{1}{t})$ at infinity (recall that $F(t)= O(t^{n-1})$).
We can thus bound with uniform positive constants $C, C'$:
\[
\begin{split}
g(JK_r, JK_r) &= \sum_{j=1}^{\ell}\left(\frac{\prod_{k=1}^{\ell}(\xi_j - \alpha_k)G(\xi_j)}{\Delta(\xi_j)}\sigma^2_{r-1}(\hat \xi_j)\right) \\
                      &\le {C \xi_{\ell}}\left(\sum_{j=1}^{\ell}\frac{\prod_{k=1}^{\ell}(\xi_j - \alpha_k)}{\Delta(\xi_j)}\right)\\
                      &=C\xi_{\ell}\left(\sigma_1 -\sum_{j=1}^{\ell}\alpha_j\right) \le C' \xi_{\ell}^2,
\end{split}\]
where to passing to the last line we have used the Vandermonde identities \eqref{Vandermonde} with respect to $(\xi_1, \ldots, \xi_{\ell})$ as well as the identity (see \cite[(98)]{ACG}) $\sum_{j=1}^\ell \frac{\xi_j^{\ell}}{\Delta(\xi_j)} = \sigma_1$.
We thus conclude that
\[ ||JK_r||_g \le C'|\xi_{\ell}|, \]
which, together with \eqref{distance}  and Lemma \ref{criterion},   yields that $JK_r$ is a complete vector field.

Similarly,  for the Killing vector fields \eqref{lifted} (which are lifts to $M^0$ of the generators of $\T^{d_j}$-actions on $\PP^{d_j}$)   we find from \eqref{k-order-ell} that there are uniform constants such that
\[||JX^j_{k}||_g \le A\xi_{\ell} - A' \le C d^g(p_0, \cdot) - C'.\]
This again shows that $JX^j_k$ are complete vector fields by Lemma \ref{criterion}.
 \end{proof}
 
 \begin{rem}\label{symmetry}  Notice that the principal $\T^{\ell}$-fibre bundle $P$ over $\prod_{j=1}^{\ell} \PP^{d_j}$ admits a natural $\prod_{j=1}^{\ell} {\rm U}(d_j+1)$-action, coming from the ${\rm U}(d_j+1)$-action on $\cO_{\PP^{d_j}}(-1)$. By \cite[Lemma 5]{ACGT1}, the GKRS metric $(g, \omega)$ is $\prod_{j=1}^{\ell} {\rm U}(d_j+1)$-invariant on $\R^{2n}\setminus \{0\}$, and hence on $(\R^{2n}, g)$ by isometric extension of the action (using  that $g$ is complete when $a\ge 0$). \end{rem}
 \begin{proof}[Proof of Theorem~\ref{thm:Cao}] The theorem follows by combining Lemmas~ \ref{irreducible}, \ref{l:extension}, \ref{l:completion}, \ref{l:biholomorphic} and Remark~\ref{symmetry} above.
\end{proof}

\subsection{Generalizing the Taub-NUT model: Proof of Theorem~\ref{thm:Taub-NUT}}\label{s:Taub-NUT}
We now consider the case when $\ell\ge 2$ and
\[(\xi_1, \ldots, \xi_{\ell}) \in \mathring{\rm D}:=(-\infty, \alpha_1)\times (\alpha_1, \alpha_2) \times \cdots \times (\alpha_{\ell-2}, \alpha_{\ell-1}) \times (\alpha_{\ell}, +\infty), \]
for given real numbers 
\[-\infty < \alpha_1< \cdots < \alpha_{\ell} < +\infty.\]
We fix $d_1, \ldots, d_{\ell-1} \in \Z_{\ge 0}$  such that
\[ n= \ell + \sum_{j=1}^{\ell-1} d_j,  \]
and let
\begin{equation}\label{pc-TN} 
p_c(t) := \prod_{j=1}^{\ell-1}(t-\alpha_j)^{d_j}, \qquad P(t):= \lambda \left(\prod_{j=1}^{\ell-1}(t-\alpha_j)^{d_j +1}\right), \qquad \lambda >0. 
\end{equation}
Now  $P(t)$ is a degree $(n-1)$ polynomial and for any real number $a$, we  obtain a degree $(\ell-1)$ polynomial (see \eqref{general-solution-ell})
\[ q(t) := \frac{P'(t) + 2a P(t)}{p_c(t)} = \lambda\left(\sum_{j=1}^{\ell-1}\left(2a(t-\alpha_j) +(d_j+1)\right) \!\!\!\!\!\! \prod_{k=1, k\neq j}^{\ell-1}(t-\alpha_k)\right)  .\]
Let
\[ F_1(t)=\cdots = F_{\ell-1}(t) := P(t), \quad F_{\ell}(t) := e^{-2at}\left(\int_{\alpha_{\ell}}^t e^{2at} q(x)p_c(x) dx\right). \]
Notice that $F_j(t), j=1, \ldots, \ell$ satisfy \eqref{ODE-F} and we have (see \eqref{general-solution-ell})
\[F_{\ell}(t) = P(t) + c e^{-2at}, \qquad c:= -e^{2a\alpha_{\ell}}P(\alpha_{\ell}).\]
 It is easily seen that $F_j(t)$ satisfy the inequalities \eqref{inequality} with respect to the data
\[ N=\ell-1, \qquad \eta_j:=\alpha_j, \, j=1, \ldots, \ell-1,  \]
and
\[ \varepsilon_j =(-1)^{\ell-j}, \, j=1, \ldots, \ell-1, \, \varepsilon_{\ell} =-1.\]
Furthermore, \eqref{Cao-signs-final} holds with our choice above.

Thus, with these data,  we obtain again the setup of \eqref{theta}-\eqref{k-order-ell} which gives rise to a K\"ahler metric on  $M^0= \mathring{\rm D} \times P$,  where $P$ is a principal $\T^{\ell}$-bundle over $B=\prod_{j=1}^{\ell-1} \PP^{d_j}, \, d_j \ge 0$ with curvature
\[ d\theta = \sum_{j=1}^{\ell} \check{\omega}^0_j\otimes v_j,  \]
where
\[v_j :=\left(\frac{2(d_j+1)}{q(\alpha_j)}\right)\Big(\alpha_j^{\ell-1}, \cdots,(-1)^{r-1}\alpha_j^{\ell-r}, \ldots, (-1)^{\ell-1}\Big).  \]
In the above,  $\check{\omega}^0_j$ denotes a Fubini--Study on $\PP^{d_j}$ of scalar curvature $2(d_j+1)d_j$ and we have set $d_{\ell}=0$, $\PP^{d_{\ell}}=\{pt\}$ and $\check{\omega}^0_{\ell}=0$. Furthermore,  $(v_j), j=1, \ldots, \ell$ is the basis of a lattice $\Gamma_{v} \subset \R^{\ell}$ such that $\T^{\ell}= \R^{\ell}/2\pi \Gamma_{v}$.

Letting $\check{\omega}_j := -\left(\frac{2(d_j+1)}{\varepsilon_j q(\alpha_j)}\right)\check{\omega}_j^0$ (which is positive definite with the choice of $\varepsilon_j$ as above), we have that $\check{\omega}_j$ is a K\"ahler--Einstein metric  on $\PP^{d_j}$ of scalar curvature $s_j=- \varepsilon_j d_j q(\alpha_j)$.  By Lemma~\ref{l:KRS-ell}, the  resulting K\"ahler metric is a steady GKRS on $M^0$.

\bigskip
A similar analysis  as in Section~\ref{s:type-C^n} leads to

\begin{lemma}\label{l:new} Given integers $\ell\ge 2, d_1, \ldots, d_{\ell-1}\ge 0$ with $n=\ell + \sum_{j=1}^{\ell-1}d_j$,  and a real constant $a$,  the above construction defines an  $(\ell-1)$-dimensional continuous family of non-isometric up to scale,   irreducible $\T^n$-invariant steady GKRS on $\R^{2n}$, compatible with the standard symplectic structure.  For $a>0$,  the corresponding K\"ahler metrics define complete steady GKRS on $\C^n$, whereas for $a=0$, they form an $(\ell-2)$-dimensional family of complete Ricci-flat (but not flat) metrics on $\C^n$. These complete metrics have an  $ {\rm U}(d_1+1) \times \ldots \times {\rm U}(d_{\ell-1}+1) \times {\rm U}(1)$ symmetry.
\end{lemma}
\begin{proof}  The first part of the proof is identical to the arguments in the proof of Lemmas~\ref{l:extension} and is left to the reader. We use \cite[Prop.~16]{ACG} to obtain the Ricci-flatness in the case $a=0$, and \cite[Prop.~17]{ACG} to rule out the flatness of the metric.  The (local)  irreducibility of the metric and the dimension of the families is established as in the proof of Lemma~\ref{irreducible}, using  that the space of hamiltonian $2$-forms  is $2$-dimensional (see Theorem~\ref{thm:CMR} and note  that $F_{\ell}(t) \neq F_1(t)$). To prove the completeness  of the metric, we proceed as in the proof of Lemma~\ref{l:completion},   using two separate estimates (which hold for $|\xi_1(p)|$  and $|\xi_{\ell}(p)|$ big enough):
\begin{equation}\label{distance-2}
 d^g(p_0, p) \ge C \xi_{\ell}(p) + C', \qquad  d^g(p_0, p) \ge -C\xi_{1}(p)+C'',\end{equation}
where each is obtained similarly to the estimate \eqref{distance}.
For showing that $JK_r$ is a complete vector field, we  compute using the specific form of $F_j(t)$ 
\[
\begin{split}
g(JK_r, &JK_r) =g(K_r, K_r) \\
&= \sum_{j=1}^{\ell}\left(\frac{F_j(\xi_j)}{p_c(\xi_j)\Delta(\xi_j)}\sigma^2_{r-1}(\hat \xi_j)\right) \\
                                            &= \lambda\sum_{j=1}^{\ell}\left(\frac{\prod_{k=1}^{\ell-1}(\xi_j - \alpha_k)}{\Delta(\xi_j)}\sigma^2_{r-1}(\hat \xi_j)\right) + \left(\frac{ce^{-2a\xi_\ell}}{p_c(\xi_\ell)\Delta(\xi_\ell)}\sigma^2_{r-1}(\hat \xi_\ell)\right).   %\\
                                            %&= \lambda + \left(\frac{ce^{-2a\xi_\ell}}{p_c(\xi_\ell)\Delta(\xi_\ell)}\sigma_{r-1}(\hat \xi_\ell)\right),
\end{split}\]
where we recall that $p_c$ is given by \eqref{pc-TN}. From the above and \eqref{distance} we obtain, outside a compact set, that there are uniform constants such that
\[ ||K_1||_g \le C, \]
and
\[\|K_r\|^2_g \le A(|\xi_1|^2 + |\xi_{\ell}|^2) \le C''(d^g(p_0, p) )^2, \, r=2, \ldots, \ell.\]
Similarly to the Cao type case, from the above and \eqref{distance-2}, we can also evaluate the square-norms of the lifted Killing fields \eqref{lifted} (when $d^g(p_0, p) \to \infty$) 
\[ ||X_k^j||^2_g \le \left|\prod_{i=1}^{\ell} (\alpha_j- \xi_i)\right| + Cd^g(p_0, p) + C' \le C'' \left(d^g(p_0, p)\right)^2, \]
for $ j=1, \ldots, \ell-1.$ We conclude again by Lemma~\ref{criterion}.
\end{proof}

\begin{proof}[Proof of Theorem~\ref{thm:Taub-NUT}] This is an immediate consequence of Lemma~\ref{l:new}.
\end{proof}

\subsection{The volume growth of riemannian balls} The complete steady GKRS's on $\C^n$ obtained in the previous lemma are not even locally isometric to the ones  given by Lemma~\ref{l:biholomorphic}. This can be seen by noticing that the corresponding hamiltonian $2$-forms have different invariants (order and number of distinct constant eigenvalues), and then using Theorem~\ref{thm:CMR} as in the proof of Lemma~\ref{irreducible}. One can also show that  when $a>0$, the soliton vector field for a metric given by Lemma~\ref{l:new}  is of the form $a\left(\sum_{i=1}^{n} \lambda_i X_i \right)$ with $\prod_{i=1}^n \lambda_i<0$  whereas the soliton vector field for a metric given by Lemma~\ref{l:biholomorphic} is of the form $a\left(\sum_{i=1}^n \lambda_i X_i \right)$ with $\lambda_i>0$. Below we reveal yet another geometric difference of the two type of GKRS's, expressed in terms of the growth of the volume of the riemannian balls $B^{g}(p_0, r)$ as $r \to \infty$.
\begin{lemma}\label{l:volume} Let $g$ be the complete riemannian metric on $\R^{2n}$,  corresponding to a steady GKRS on $\C^n$ which is either of Cao type, i.e. obtained by the construction of Lemma~\ref{l:completion} with $a>0$,  or  is of Taub-NUT type, i.e. obtained by the construction of Lemma~\ref{l:new} with $a\ge 0$. Denote by $V^g(r)$ the volume with respect to $g$ of the riemannian ball $B^g(p_0, r)$ of radius $r>0$ centered at some base point $p_0 \in \R^{2n}$. Then, as $r\to \infty$, $V^g(r) \sim r^{n}$ in the Cao type case and $V^{g}(r) \sim r^{2n-1}$ in the Taub-NUT type case.
\end{lemma}
\begin{proof} As before, we denote by $d^g(p_0, p)$ the riemannian distance and by $\xi_j(p)$ the values of the global smooth functions $\xi_j$ on $M=\R^{2n}$. To simplify our argument, we shall (without loss) establish the result  on the dense subset $M^0$ of $M=\R^{2n}$, where we recall  $M^0=\mathring{{\mathrm D}} \times P,$ where $P \to \prod_{j=1}^{\ell} \PP^{d_j}$ is a principal $\T^\ell$-bundle and $\mathring{\rm D}=(\alpha_1, \alpha_2) \times \cdots \times (\alpha_{\ell-1}, \alpha_{\ell})$ in the Cao type case and $\mathring{\rm D}=(-\infty, \alpha_1) \times (\alpha_1, \alpha_2) \times \cdots \times (\alpha_{\ell-2}, \alpha_{\ell-1}) \times (\alpha_{\ell}, + \infty)$ in the Taub-NUT case. We denote as above by $d^g(p_0, p)$ the riemannian distance between a point $p\in M^0$ and a fixed base point $p_0\in M^0$, and by $\xi_j(p)$ the corresponding coordinate of $p$ on the $\mathring{\rm D}$-factor. 

We start by establishing the following asymptotic behaviour when $d^g(p_0, p) \to \infty$:
\begin{equation}\label{distance-Cao}
d^g(p_0, p) \sim \xi_{\ell}(p) \, \, \textrm{in the Cao type case}, 
\end{equation}
and
\begin{equation}\label{distance-Taub-NUT}
\qquad d^g(p_0, p) \sim \xi_{\ell}(p) - \xi_1(p) \, \, \textrm{in the Taub-NUT case}.
\end{equation}

We first consider the Cao type case. The asymptotic \eqref{distance-Cao} means that there  is a  uniform positive constant $\Lambda>0$ such that 
\[ \frac{1}{\Lambda} \xi_{\ell}(p) \le  d^g(p_0, p) \le \Lambda \xi_{\ell}(p) \]
holds when $d^g(p_0,p)$ is sufficiently large.  The inequality on the LHS follows from \eqref{distance}. 

To establish the inequality at the RHS, let  $p=(\xi, p')\in M^0=\mathring{\rm D} \times P$ and  $p_0=(\xi^0, p'_0)$.  Without loss, we assume $\xi_{\ell}^0>0$.  As $\xi_1, \ldots, \xi_{\ell-1}$ are continuous and bounded  and $\sigma_1= \xi_1 + \cdots +\xi_{\ell}$ is proper on $M=\R^{2n}$, we can also assume (by taking $d(p_0, p)$ be large enough) that $\xi_{\ell} > \xi_{\ell}^0>0$. Let $\tilde p_0=(\xi^0, {\tilde p}'_0)$ be a point in the $\T^{\ell}$-orbit of $p_0$,  such that $p'$ and ${\tilde p}_0'$ are joined by the horizontal lift $c(t)$ of a minimal geodesic on $\prod_{j=1}^{\ell} (\PP^{d_j}, \check{\omega}_j^0)$; here, we endow the principal $\T^{\ell}$-bundle $P \to \prod_{j=1}^{\ell} \PP^{d_j}$  with the invariant metric 
\[ g_P := \sum_{j=1}^{\ell} \left( \check{\omega}^0_{j} + \theta_j \otimes \theta_j\right).\]
We thus have
\[ d^g(p_0, p) \le d^g({\tilde p}_0, p) + {\rm diam}(\T^{\ell} \cdot p_0). \]
To establish the desired inequality  (for all $p$ such that $d^g(p_0, p)$ and hence $\xi_{\ell}(p)$ is big enough), we can therefore assume without loss that $p_0 = {\tilde p}_0$.  Let $\gamma(t)= (t\xi + (1-t) \xi^0, c(t))$ be a curve joining 
$p=(\xi_{\ell}, p')$ and $p_0=(\xi_{\ell}^0, p_0'), \, t\in [0, 1]$, where $c(t)$ is a horizontal geodesic on $P \to \prod_{j=1}^{\ell} (\PP^{d_j}, \check{\omega}_j^0)$.
We use the form \eqref{k-order-ell} of the metric $g$ with $F_j(t) =F(t) = \prod_{j=1}^{\ell}(t- \alpha_j)^{d_j+1} G(t),$ where $G(t)$ is analytic and strictly positive on $[\alpha_1, \infty)$. 
We denote by $M>0$ the minimum of $G(t)$ on $[\alpha_1, \alpha_{\ell}]$ and notice that for any $j=1, \ldots, \ell-1$
%~\footnote{CC: Should the second term instead be $ \left|\frac{\prod_{k\neq j}(\xi_j - \xi_k)}{\prod_{k=1}^\ell (\xi_j - \alpha_k) G(\xi_j)}\right|$?}
\[\left|\frac{p_c(\xi_j)\Delta(\xi_j)}{F(\xi_j)}\right| = \left|\frac{\prod_{k\neq j}(\xi_j - \xi_k)}{\prod_{k=1}^\ell (\xi_j - \alpha_k) G(\xi_j)}\right| \le \frac{C}{M}\left|\frac{1}{(\xi_j-\alpha_j)}\right|,\]
where the constant $C$ depends only on the $\alpha_j's$. 

As the polynomial $q(t)$ in the definition \eqref{F} of $F(t)$ is of degree $(\ell-1)$ (recall that $q$ has a unique root in each interval $(\alpha_i, \alpha_{i+1})$), $F(t)= P(t) + ce^{-at}$ where $P(t)$ is of polynomial of degree $(n-1)$. It thus follows that there exists a uniform constant $N>0$ such that $\left| F(t)/t^{n-1} \right| \ge N$%\footnote{\textcolor{red}{CC: I'm not sure I fully understand the referee's (referee Q) concern here. I the point that we really want to make here is that $|F(t)|/|t|^{n-1}$ is uniformly bounded from below for all $t$ such that $|t|$ is sufficiently large. This then applies equally well to the Taub-NUT case where I guess it's not always true that $F(t)/t^{n-1} > 0$ for $t < \alpha_1$. Perhaps we should just replace $F(t)/t^{n-1}$ with $\left| F(t)/t^{n-1} \right|$?  }} 
when $|t|$ is sufficiently big. It follows that for $d^g(p, p_0)$ (and hence also $\xi_{\ell}$) large enough
\[\left|\frac{p_c(\xi_\ell)\Delta(\xi_\ell)}{F(\xi_\ell)}\right| \le \frac{C'}{N}, \]
where $C'$ is a uniform constant depending only on $\alpha_j's$.

From the above  and  \eqref{k-order-ell} we deduce that for  $d^g(p_0, p)$ (and hence also $\xi_{\ell}(p))$ large enough
\[
\begin{split} d^{g}(p_0, p) &\le L^g(\gamma(t))) \le D \sqrt{(\xi_{\ell}-\alpha_{1})} +  2\sqrt{\frac{C}{M}}\sum_{j=1}^{\ell-1}\sqrt{|\xi_j- \alpha_j|} + \frac{C'}{N}\xi_{\ell} \\
& \le D \sqrt{(\xi_{\ell}-\alpha_{1})} +  2\sqrt{\frac{C}{M}}\sum_{j=1}^{\ell-1}\sqrt{|\alpha_{j+1}- \alpha_j|} + \frac{C'}{N}\xi_{\ell}  \le \Lambda \xi_{\ell}.
\end{split} \]
 The term $D \sqrt{(\xi_{\ell}-\alpha_{1})}$ above comes from the calculation of $L^{g}(\gamma(t))$ along the horizontal part of the metric. Here 
 $D$  is a uniform positive constant which depends on the $\alpha_j's$ and the diameter of $\prod_{j=1}^{\ell}(\PP^{d_j}, \check{\omega}_j^0)$.   The  two other terms in the inequality above come from the vertical component of the metric.

\bigskip
The asymptotic \eqref{distance-Taub-NUT} is  proved similarly, with some straightforward modification in the arguments, but the last inequalities yield
%\[d^{g}(p_0, p)  \le D\sqrt{(\xi_{\ell}-\alpha_{1})(\alpha_{\ell}-\xi_1)}  + 2\sqrt{\frac{C}{M}}\sum_{j=2}^{\ell-1}\sqrt{|\alpha_{j+1}- \alpha_j|} + \frac{C'}{N}(\xi_{\ell} - \xi_1).  \]
\[d^{g}(p_0, p)  \le D\sqrt{(\xi_{\ell}-\alpha_{1})(\alpha_{\ell}-\xi_1)} + \frac{C'}{N}(\xi_{\ell} - \xi_1) + C'',  \]
where $C'' = 2\sqrt{\frac{C}{M}}\sum_{j=2}^{\ell-1}\sqrt{|\alpha_{j+1}- \alpha_j|}$. Applying the Cauchy-Schwarz inequality to the first summand on the RHS  leads to the desired estimate $d^{g}(p_0, p) \le \Lambda(\xi_{\ell}-\xi_1)$. 

\bigskip
It now follows from \eqref{distance-Cao}  and the form of the symplectic form $\omega$ in \ref{k-order-ell} that,  in the Cao type case,  the volume of the riemannian ball $B^g(0, r)$  of radius $r$ centred at the origin of $\C^n$ grows  when $r\to \infty$ as the volume of ${\rm D}_r:=[\alpha_1, \alpha_2]\times \cdots \times [\alpha_{\ell-1},\alpha_{\ell}]\times [\alpha_{\ell}, r]$ with respect to the measure 
\begin{equation}\label{volume-ratio} \left|\det\left(V(\xi_1, \ldots, \xi_{\ell})\right)\left(\prod_{j=1}^\ell p_c(\xi_j)\right)\right|d\xi_1\ldots d\xi_{\ell},\end{equation}
where $V(\xi_1, \ldots, \xi_\ell)$ is the Vandermonde  matrix (see e.g. \cite[App.~B]{ACG}).  Expanding with respect to $\xi_{\ell}$ in the Cao type case the  polynomial quantity
%\[ \det V(\xi_1, \ldots, \xi_{\ell})\left(\prod_{j=1}^\ell p_c(\xi_j)\right) =\xi_\ell^{n-1} \det V(\xi_1, \ldots, \xi_{\ell-1}) +  \, \textrm{lower  order} , \]
\[\begin{split} W(\xi_1, \ldots, \xi_\ell) &:= \det V(\xi_1, \ldots, \xi_{\ell})\left(\prod_{j=1}^\ell p_c(\xi_j)\right) \\
&=\xi_\ell^{n-1} \det V(\xi_1, \ldots, \xi_{\ell-1}) +  \, \textrm{lower  order} , \end{split} \]
we get the rate $\sim r^n$ for the volume  ${\rm D}_r$ and hence of $B^g(0, r)$. In the Taub-NUT type case,  we have that the volume of $B^g(0, r)$ grows at the rate of the volume of ${\rm D}_{-r,r}:=[-r, \alpha_1]\times [\alpha_1, \alpha_2] \times \cdots \times [\alpha_{\ell-2}, \alpha_{\ell-1}]\times [\alpha_{\ell}, r]$ with respect to the same measure \eqref{volume-ratio}.  In this case, we expand with respect to the absolute degree of the  two variables $\xi_1$ and $\xi_{\ell}$:
%\[ (-1)^{\ell-1}\det V(\xi_1, \ldots, \xi_{\ell})\left(\prod_{j=1}^\ell p_c(\xi_j)\right) = (\xi_{1}^{n-1}\xi_{\ell}^{n-2} -\xi_1^{n-2}\xi_{\ell}^{n-1}) \det V(\xi_2, \ldots, \xi_{\ell-1})+   \, \textrm{lower order} \]
\[\begin{split} W &:= (-1)^{\ell-1}\det V(\xi_1, \ldots, \xi_{\ell})\left(\prod_{j=1}^\ell p_c(\xi_j)\right) \\
&=(\xi_{1}^{n-1}\xi_{\ell}^{n-2} -\xi_1^{n-2}\xi_{\ell}^{n-1}) \det V(\xi_2, \ldots, \xi_{\ell-1})+   \, \textrm{lower order} , \end{split} \]
which gives, after integration,  a volume rate $\sim r^{2n-1}$. \end{proof}

% \begin{rem}\label{r:Scal-decay} According to \cite[(79)]{ACG}, the scalar curvature of the metrics is given by the formula \[ \Scal_g = \sum_{j=1}^\ell\left( \frac{\varepsilon_j d_j q(\alpha_j)}{\prod_{k=1}^{\ell}(\alpha_j -\xi_k)} - \frac{F_j''(\xi_j)}{p_c(\xi_j) \Delta(\xi_j)}\right).\] Using that in the Cao type case $F_j(t)=F(t)= P(t) + c e^{-2at}$ for a degree $(n-1)$ polynomial $P(t)$ and \eqref{distance-Cao} we obtain a decay of $\Scal_g $  to $0$ when $d^g(p_0, p) \to \infty$ at the rate of $\frac{1}{d^g(p_0, p)}$.~\footnote{VA: Is this consistant with Conlon-Deruelle ?} \end{rem}

\section{The 2-dimensional case}\label{s:D2} Here we specialize our previous constructions to $n=2$ in order to give a concise picture for the complete steady GKRS's on $\C^2$ we have found.
 
 \subsection{Extending Cao's soliton on $\C^2$: Proof of Corollary~\ref{c:Cao-conjecture}}\label{ss:Cao2D} We consider first the examples  on $\C^2$ given by the deformation of the flat model as in Lemma~\ref{l:extension}. We can take either $\ell=1, d_1=1$ or $\ell=2, d_1=d_2=0$. In the first case, we obtain Cao's soliton on $\C^2$ described in Section~\ref{s:Cao}. In the latter case, we have an additional $1$-parameter family (up to scale) of complete steady K\"ahler-Ricci solitons,  parametrized by a positive constant $a>0$,  and given by the \emph{orthotoric} ansatz:
 \begin{equation}\label{Cao-extension}
 \begin{split}
 g= & \frac{(\xi_1-\xi_2)}{F(\xi_1)} d\xi_1^2 + \frac{(\xi_2-\xi_1)}{F(\xi_2)}d\xi_2^2 \\
       & + \frac{F(\xi_1)}{(\xi_1-\xi_2)}\left(dt_1 + \xi_2 dt_2\right)^2 + \frac{F(\xi_2)}{(\xi_2-\xi_1)}\left(dt_1 + \xi_1 dt_2\right)^2,\\
  \omega = & d\xi_1\wedge (dt_1+ \xi_2 dt_2)  + d\xi_2 \wedge (dt_1 + \xi_1dt_2).
  \end{split} \end{equation}
  with
  \[ F(t) := (1- e^{-2a}) t - (1  - e^{-2at}). \] 
The above metric is defined on $\mathring{\rm D}   \times \T^2$ with
\[  (\xi_1, \xi_2) \in \mathring{\rm D}  =(0, 1) \times (1, \infty). \]
Note that we can use the definition of $v_j$ \eqref{vj-def} to write the standard $\T^2$-action on $\C^2$ in terms of the orthotoric data above. Proceeding in this way, we see that the generators $X_1, X_2$ for the standard $\T^2$-action are given by
\[ X_1= -\left(\frac{2}{F'(0)}\right) \frac{\partial}{\partial t_2}, \qquad X_2 = \left(\frac{2}{F'(1)}\right)\left(\frac{\partial}{\partial t_1} - \frac{\partial}{\partial t_2}\right) \]
%are circle groups generators of $\T^2$; letting $\sigma_1= \xi_1 + \xi_2$ and $\sigma_2=\xi_1\xi_2$,
and the standard momentum coordinates are 
\[\mu_1 =- \left(\frac{2}{F'(0)}\right) \sigma_2, \qquad \mu_2 = \left(\frac{2}{F'(1)}\right)\left(\sigma_1 - \sigma_2\right),\]
where as usual $\sigma_1 = \xi_1 + \xi_2$ and $\sigma_2 = \xi_1\xi_2$. Our analysis in the previous sections shows that the K\"ahler metric $(g, \omega)$ extends smoothly on $\R^4$  and defines a  complete steady GKRS with a soliton vector field %~\footnote{VAS: can we specify the coefficients $F'(0)$ and $F'(1)$?}
\begin{equation}\label{vector-Cao}
\begin{split}
X&=   2a\frac{\partial}{\partial t_1}= a \left(-F'(0)X_1 + F'(1) X_2\right) \\
  &= a(e^{-2a} + 2a-1)X_1 + a(1-e^{-2a}(1+2a)) X_2.
\end{split}
\end{equation}
As we saw in Section \ref{s:type-C^n}, we can obtain other solutions by multiplying $F$ by a constant scale factor $c > 0$.  
\begin{rem}\label{specifyvf-Cao2d}
	Running through all choices of $a, c > 0$,  the corresponding soliton vector fields associated to the solutions for these choices exhaust the region 
	\[ X = \lambda_1 X_1 + \lambda_2 X_2, \]
	for $\lambda_1 > \lambda_2 > 0$.  By the $\Z_2$-symmetry of $\C^2$ interchanging the two axes fixed by the standard torus action, we therefore get all $X$ with $\lambda_1 \neq \lambda_2$, $\lambda_1, \lambda_2 > 0$. 

Indeed, observe that 
\[  \frac{\lambda_1}{\lambda_2} = \frac{e^{-2a} - 1 + 2a}{1 - e^{-2a}(1+2a)}, \]
which is strictly increasing in $2a$, and satisfies 
\[ \lim_{a \to \infty} \frac{\lambda_1}{\lambda_2}  = \infty, \qquad \lim_{a \to 0}  \frac{\lambda_1}{\lambda_2}  = 1. \]
In particular, for fixed $c > 0$ we can modify $a$ to obtain any ratio $\frac{\lambda_1}{\lambda_2} > 1$. Hence by adjusting $c$ we get all $X$ as claimed.
\end{rem}
By Lemma~\ref{l:biholomorphic}, the corresponding complex structure $J$ is equivariantly isomorphic to the standard complex structure $J_0$ on $\C^2$. To see this last point more directly, we consider the functions
\[ 
\begin{split}
y_1(\xi_1, \xi_2) :=& \int^{\xi_1} \frac{t}{F(t)} dt  + \int^{\xi_2} \frac{t}{F(t)} dt, \\
 y_2(\xi_1, \xi_2) :=&- \int^{\xi_1} \frac{1}{F(t)} dt - \int^{\xi_2} \frac{1}{F(t)} dt, \end{split} \]
 which are pluriharmonic with respect to $J$ (see \cite[p.~391]{ACG}). The fact that $J\frac{\partial}{\partial t_i}$ are complete vector fields tells us that $(y_1, y_2)$ sends $\mathring{\rm D}$  diffeomorphically to $\R^2$ (a fact that can be easily checked explicitly, using that $F(t)$ is an entire analytic function with simple zeroes at $0$ and $1$). Indeed, in the local (holomorphic) coordinates $(y_1+ \sqrt{-1} t_1, y_2 + \sqrt{-1} t_2)$, the vector fields $J\frac{\partial}{\partial t_i}$  become $\frac{\partial}{\partial y_i}$. Notice that the pluriharmonic functions
 \[
 \begin{split}
v_1(\xi_1, \xi_2) :=&  -\left(\frac{F'(0)}{2}\right)\left(\int_1^{\xi_1} \frac{t-1}{F(t)} dt  + \int_1^{\xi_2} \frac{t-1}{F(t)} dt\right), \\
v_2(\xi_1, \xi_2) :=&  +\left(\frac{F'(1)}{2}\right)\left( \int_0^{\xi_1} \frac{t}{F(t)} dt  +\int_0^{\xi_2} \frac{t}{F(t)} dt\right), \end{split} \]
allow us to identify the generators $JX_1$ and $JX_2$ of the $\C^{\times}\times \C^{\times}$-action with $\frac{\partial}{\partial v_1}$ and $\frac{\partial}{\partial v_2}$. 
Letting
\[ e^{2v_1} = |z_1|^2, \qquad e^{2v_2}= |z_2|^2, \]
determines implicitly $(\xi_1, \xi_2)$ as functions of $(|z_1|^2, |z_2|^2)$.

Finally, using that $F(t)$ is an analytic function  with single roots at $0$ and $1$,  letting 
\[ G(t) := \frac{F(t)}{t(t-1)}\]
defines a smooth positive function on $\R$. According to \cite[Prop.~11]{ACG}, the smooth function on $\R^2$
\[H(\xi_1, \xi_2):= \int^{\xi_1} \frac{dt}{G(t)} + \int^{\xi_2} \frac{dt}{G(t)} \]
 is a global K\"ahler potential of $\omega$, i.e. $\omega = dd^c H$. Using that $(\xi_1, \xi_2)$ extend as smooth functions to $\R^4$ (this is because the K\"ahler metric extends smoothly and hence the corresponding hamiltonian $2$-form with eigenvalues $(\xi_1, \xi_2)$ also does) $H$ is a global K\"ahler potential. By the above identification,  $H(\xi_1, \xi_2)$ can also be viewed as a smooth function $\Phi(|z_1|^2,|z_2|^2)$ on $\C^2$  with  $\omega = dd^c \Phi$.
 
 \begin{proof}[Proof of Corollary~\ref{c:Cao-conjecture}] The explicit form \eqref{Cao-extension} of the metric (with $F(t)=(1- e^{-2a}) t - (1  - e^{-2at}), \, a>0$) makes it straightforward to perform curvature computations.  It turns out that:
 \begin{enumerate}
 \item[$\bullet$] the sectional curvature of \eqref{Cao-extension} is positive;
 \item[$\bullet$] the scalar curvature of \eqref{Cao-extension} achieves its maximum at the origin of $\C^2$.
 \end{enumerate}
 For the reader's convenience, we provide the detailed verifications of the above facts in  Propositions~\ref{Ric>0}, \ref{scal-max} and \ref{full-sec} in the Appendix.
 \end{proof}

\subsection{Deforming the Taub-NUT metric on $\C^2$: Proof of Corollary~\ref{c:Taub-NUT}}  This the case considered in Sec.~\ref{s:Taub-NUT} with $\ell=2$ and $d_1=0$. According to Lemma~\ref{l:new}, we have a $1$-parameter family of toric steady GKRS on $\C^2$, parametrized by a non-positive constant $a\ge 0$, which are again given by the orthotoric ansatz
\begin{equation}\label{Taub-NUT-extension}
 \begin{split}
 g= & \frac{(\xi_1-\xi_2)}{F_1(\xi_1)} d\xi_1^2 + \frac{(\xi_2-\xi_1)}{F_2(\xi_2)}d\xi_2^2 \\
       & + \frac{F_1(\xi_1)}{(\xi_1-\xi_2)}\left(dt_1 + \xi_2 dt_2\right)^2 + \frac{F_2(\xi_2)}{(\xi_2-\xi_1)}\left(dt_1 + \xi_1 dt_2\right)^2,\\
  \omega = & d\xi_1\wedge (dt_1+ \xi_2 dt_2)  + d\xi_2 \wedge (dt_1 + \xi_1dt_2).
  \end{split} \end{equation}
  with
  \begin{equation}\label{taubnut-F} F_1(t) := t, \qquad F_2(t) := t -e^{2a}e^{-2at}. \end{equation}
The metric is now defined on $\mathring{\rm D}   \times \T^2$ with
\[  (\xi_1, \xi_2) \in \mathring{\rm D}  =(-\infty, 0) \times (1, \infty), \]
and 
\[ X_1= -\left(\frac{2}{F_1'(0)}\right) \frac{\partial}{\partial t_2}= -2\frac{\partial}{\partial t_2}, \]
\[ X_2 = \left(\frac{2}{F_2'(1)}\right)\left(\frac{\partial}{\partial t_1} - \frac{\partial}{\partial t_2}\right)=\frac{2}{1+2a}\left(\frac{\partial}{\partial t_1} - \frac{\partial}{\partial t_2}\right) \]
being the circle group generators of $\T^2$.

The K\"ahler metric \eqref{Taub-NUT-extension} extends smoothly on $\R^4$  and defines a  complete steady GKRS with a soliton vector field 
\begin{equation}\label{vector-Taub-NUT}
X=  2a\frac{\partial}{\partial t_1}=  a \left(-X_1 + (1+2a) X_2\right). \end{equation} 
Furthermore, by Lemma~\ref{l:biholomorphic}, the corresponding complex structure $J$ is equivariantly isomorphic to the standard complex structure $J_0$ on $\C^2$. Here again we can build the metric on $\C^2$  by considering the pluriharmonic functions
 \[
 \begin{split}
v_1(\xi_1, \xi_2) :=&  -\frac{F_1'(0)}{2}\left(\int_1^{\xi_1} \frac{t-1}{F_1(t)} dt  + \int_1^{\xi_2} \frac{t-1}{F_2(t)} dt\right), \\
                             =&  -\frac{1}{2}\left( \xi_1 - 1 - \log|\xi_1| + \int_1^{\xi_2} \frac{t-1}{F_2(t)} dt\right) \\
v_2(\xi_1, \xi_2) :=& \frac{F_2'(1)}{2}\left( \int_0^{\xi_1} \frac{t}{F_1(t)} dt  +\int_0^{\xi_2} \frac{t}{F_2(t)} dt\right) \\
                           =& \frac{(2a+1)}{2}\left( \xi_1  + \int_0^{\xi_2} \frac{t}{F_2(t)} dt\right).\end{split} \]
Letting
\[ e^{2v_1} = |z_1|^2, \qquad e^{2v_2}= |z_2|^2, \]
determines implicitly $(\xi_1, \xi_2)$ as functions of $(|z_1|^2, |z_2|^2)$.  
Furthermore, a  global K\"ahler potential is given in this case by
\[ 
\begin{split}
H(\xi_1, \xi_2)=& \int^{\xi_1} \frac{t(t-1)}{F_1(t)} dt + \int^{\xi_2}\frac{t(t-1)}{F_2(t)} dt \\
                       =& \frac{1}{2} (\xi_1-1)^2 + \int^{\xi_2}\frac{t(t-1)}{F_2(t)} dt . \end{split} \]        
             As $F_2(t)$ is analytic and monotone with simple zero at $1$, $H$ is globally defined on $\R^2$, thus giving rise to a global smooth function $\Phi(|z_1|^2, |z_2|^2)$ on $\C^2$ such that $\omega = dd^c H = dd^c \Phi$.    

\begin{proof}[Proof of Corollary~\ref{c:Taub-NUT}] The above construction applies for $a=0$,  in which case $F_2(t)= (t-1)$ and our calculation gives
\[ \begin{split}
v_1(\xi_1, \xi_2) =& -\frac{1}{2}\Big( \xi_1 + \xi_2 - 2 - \log|\xi_1|\Big) \\
v_2(\xi_1, \xi_2) = & + \frac{1}{2}\Big( \xi_1 + \xi_2  + \log|\xi_2-1|\Big) \\
|z_1|^2 = & -e^2 e^{-(\xi_1 + \xi_2)} \xi_1, \\
|z_2|^2 = & e^{(\xi_1 + \xi_2)} (\xi_2 - 1), \\
H(\xi_1, \xi_2)= & \frac{1}{2}\Big( (\xi_1-1)^2  + \xi_2^2\Big),
\end{split} \]
which easily connects to the computation in \cite{LeBrun} expressing the Taub-NUT metric.  Alternatively, we can refer to \cite{Gauduchon-Taub-NUT} which identifies \eqref{Taub-NUT-extension} for $F_1(t)=t, \, F_2(t)= (t-1)$ with the Taub-NUT metric. \end{proof} 

\begin{rem}\label{r:2dvectors}
As usual, we identify a point $X = \lambda_1 X_1 + \lambda_2 X_2$ with $(\lambda_1, \lambda_2) \in \R^2$.  Then,  similarly to Remark~\ref{specifyvf-Cao2d}, we can see from the expression \eqref{vector-Taub-NUT} for the soliton vector field that this family gives examples which fill the regions 
\[ C_+ = \left\{ (\lambda_1, \lambda_2) \: | \: \lambda_2 > - \lambda_1 > 0\right\}, \qquad C_- = \left\{ (\lambda_1, \lambda_2) \: | \: \lambda_1 > - \lambda_2 > 0\right\},  \]
after accounting for rescaling and symmetry.  Together with the examples of Section~\ref{ss:Cao2D} with $\ell = 2$, $d_1 = d_2 = 0$ and with $\ell = 1$, $d_1=1$ (Cao's examples),  this fills the region 
\[ \mathcal{C} := \{ (\lambda_1, \lambda_2) \: | \: \lambda_1 + \lambda_2 > 0, \, \lambda_i \neq 0 \},   \]
in the sense that for each $X = (\lambda_1, \lambda_2) \in \mathcal{C}$, there is a complete indecomposable steady GKRS with soliton vector field $X$.  Allowing for products, we can of course fill in the rays $\lambda_1 = 0$,  $\lambda_2 \geq 0$,  and $\lambda_2 = 0$,  $\lambda_1 \geq 0$ by taking products of the cigar soliton and the flat metric on $\C$.  Notice that this does not quite fill the complement of the forbidden region 
\[ \R^2 \backslash \! -C^*(\Pol) = \{ (\lambda_1, \lambda_2) \: | \: \lambda_1 \geq 0 \text{ or } \lambda_2 \geq 0 \}.\]
\end{rem}

\appendix
\section{Curvature estimates of the Cao type examples on $\C^2$}\label{s:appendix}
\subsection{The lower bound of the sectional curvatures of an ambi-K\"ahler $4$-manifold} In this subsection, we spell out the condition for the positivity of the holomorphic sectional curvature on a class of K\"ahler complex surfaces to which our examples of GKRS belong.

\begin{dfn}~\cite{ACG2} Let $M$ a real $4$-dimensional manifold endowed with a K\"ahler structure $(g, J)$. We say that $g$ is \emph{ambi-K\"ahler} if there exists a second $g$-orthogonal complex structure $I$ which induces the opposite orientation to the one induced by $J$, and a positive function $f$ such that the conformal metric $\tilde g = \frac{1}{f^2} g$ is K\"ahler with respect to $I$. We say that the Ricci tensor $Ric$ of $g$ is \emph{ambi-hermitian} if it is both $I$ and $J$ invariant, i.e.
\[Ric(JX, JY) = Ric(IX, IY) = Ric(X, Y).\]
\end{dfn}
Ambi-K\"ahler $4$-manifolds with ambi-hermitian Ricci tensor have been studied and classified in \cite{ACG2}. They include as a special case the metrics obtained by the orthotoric ansatz, and in particular the GKRS examples in Section~\ref{s:D2}. 

Being $J$-invariant,  the traceless Ricci tensor $Ric_0$ has at each point two real eigenvalues $(\lambda, -\lambda)$,  each of multiplicity at least $2$. Indeed, at any point where $Ric_0$ is non-zero, the eigenspace decomposition of $Ric_0 \circ g^{-1}$ gives rise to a $g$-orthogonal splitting $TM= T_+M \oplus T_-M$, where $T_{\pm}M$ are of real dimension $2$ (as being $J$-invariant) and correspond to the eigenvalues $\lambda$ and $-\lambda$  of $Ric_0$; furthermore, as $Ric_0$ is also $I$-invariant, $T_{\pm}M$ are $I$-invariant. As $I$ and $J$ yield opposite orientations, we can assume (without loss) that $I=J$ on $T_+M$ and $I=-J$ on $T_-M$. We shall also assume $\lambda \ge 0$, i.e. $\lambda = \frac{1}{4}\left\|Ric_0\right\|_g$.  It follows from the above that if we denote by $\omega_J = gJ$ and $\omega_I =gI$ the  K\"ahler forms of $(g, J)$  and $(g, I)$, respectively, then
\begin{equation}\label{eq:Ricci-diagonal}
Ric_0J = \lambda \omega_I, \qquad Ric_0 I= \lambda \omega_J. \end{equation}
The \emph{conformal scalar curvature} $\kappa_I$ of the Hermitian metric $(g, I)$ is defined by
\[ \kappa_I := 3\left\langle W(\omega_I), \omega_I \right\rangle_g, \]
where $W$ denotes the Weyl  curvature tensor acting on $2$-forms via the riemannian duality $g: \bigwedge^2 TM \cong \Lambda^2 T^*M$. Notice that by~\cite{AG}, for the K\"ahler structure $(g, J)$ we have 
\[ \kappa_J = 3\left\langle W(\omega_J), \omega_J \right\rangle_g=Scal_g\] is the scalar curvature whereas for the conformally K\"ahler hermitian structure $(g, I)$ we get
\[  \kappa_I = \frac{1}{h^2} Scal_{\tilde g},\]
where $Scal_{\tilde g}$ is the scalar curvature of the conformal K\"ahler metric $(\tilde g= \frac{1}{h^2} g, I)$.
We now state our main technical result
\begin{lemma}\label{sec} Suppose  $(M, g, J)$ is ambi-K\"ahler  with respect to the complex structure $I$, and has ambi-hermitian Ricci tensor. Then, in the notations above,  at any given point, the minimum of the sectional curvature of $g$ is given by the minimum of  the function%~\footnote{There is now a discrepancy by a factor $\frac{1}{24}$  in the various definitions of $f$. CC: fixed the factor of $1/24$, and standardized the notation for the various sectional curvatures}
\[f_s(t_1, t_2):= t_1^2 \left(\frac{\Scal_g }{8}\right)+ t_1t_2 \lambda + t_2^2\left(\frac{\kappa_I}{8}\right) + \left(\frac{\Scal_g - \kappa_I}{24}\right), \]
for $(t_1, t_2)\in [-1,1]\times [-1,1]$. Furthermore, the minimum of the holomorphic sectional curvature of $(g, J)$ is given by the minima of 
\[f(t) := f_s(1, t)= \frac{\Scal_g}{6} - \frac{\kappa_I}{24} + \left(\frac{\kappa_I}{8}\right)  t^2 + \lambda t, \qquad t\in [-1, 1]. \]
\end{lemma}
\begin{proof}
We consider the orthogonal splitting
\begin{equation}\label{Hodge}
 \Lambda^2M = \Lambda^2_+M \oplus \Lambda^2_-M \end{equation}
of $2$-forms in $4$-dimensions as a direct sum of self-dual and anti-self-dual forms. We have furthermore the orthogonal decompositions (see e.g. \cite{AG})
\begin{equation}\label{types}
 \Lambda^2_+M= \R\langle \omega_J \rangle \oplus \Lambda^{(2,0)+(0,2)}_J M, \qquad \Lambda^2_-M= \R\langle \omega_I \rangle \oplus \Lambda^{(2,0)+(0,2)}_I M, \end{equation}
where $\Lambda^{(2,0)+(0,2)}_J M$ (resp. $\Lambda^{(2,0)+(0,2)}_I M$) denote the real bundles of $2$-forms of type $(2,0)+(2,0)$ with respect to $J$ (resp. $I$).

 Using the metric $g$, we shall consider the riemannian curvature as a self-adjoint operator $R : \Lambda^2M \to \Lambda^2M$ with respect to the induced inner product $\langle \cdot, \cdot \rangle_g$ on $\Lambda^2M$.  Also,  we shall tacitly identify decomposable $2$-vectors $X\wedge Y$ with the corresponding $2$-forms $\alpha = g(X) \wedge g(Y)$ satisfying $\alpha\wedge \alpha =0$. Conversely, a little linear algebra shows that a $2$-form $\alpha$ corresponds to a decomposable $2$-vector if and only if $\alpha \wedge \alpha=0$. Thus,  at any given point, the sectional curvature $K(\alpha)$ is equivalently expressed as
\[ K(\alpha)= \langle R(\alpha), \alpha \rangle_g, \qquad \alpha \in \Lambda^2M, \, ||\alpha||_g=1, \, \alpha \wedge \alpha =0. \]
Using \eqref{Hodge}, we can decompose $\alpha= \alpha_+ + \alpha_-$ with $\alpha_+ \in \Lambda^2_+M$ and $\alpha_- \in \Lambda^2_-M$, and thus the above relation is equivalent to 
\begin{equation}\label{4D-sectional}
K(\alpha)=\langle R(\alpha_+ + \alpha_-), \alpha_+ + \alpha_- \rangle_g, \qquad  \alpha_\pm  \in \Lambda^2_\pm M, \, ||\alpha^{\pm}||^2_g = \frac{1}{2}.\end{equation}
Recall that  the curvature operator can be written as
\begin{equation}\label{Singer-Thorpe}
R = \frac{Scal_g}{12} Id + \widehat{Ric_0}  + W^+ + W^-, \end{equation}
where  $W^{\pm}$ are respectively the self-dual and anti-self-dual Weyl  curvature tensors, obtained by restricting the Weyl tensor to $\Lambda^{2}_+M$ and $\Lambda^{2}_-M$, and 
\[ \widehat{Ric_0}(X\wedge Y) := \frac{1}{2}\left(Ric_0(X) \wedge Y + X \wedge Ric_0(Y)\right), \]
with $\widehat{Ric_0} : \Lambda^{2}_{\pm}M  \to \Lambda^2_{\mp}M$.

It is well known (see~\cite{AG}) that in the ambi-K\"ahler situation, the decompositions \eqref{types} also coincide  with the respective eigenspace decompositions of $W^{+}$ and $W^{-}$, for the eigenvalues  $\left(\frac{\Scal_g}{6}, -\frac{Scal_g}{12}, -\frac{Scal_g}{12}\right)$  of $W^+$ and $\left(\frac{\kappa_I}{6}, -\frac{\kappa_I}{12}, -\frac{\kappa_I}{12}\right)$ of $W^-$. Furthermore,  we get from \eqref{eq:Ricci-diagonal} and the fact that $Ric_0$ is both $I$ and $J$-invariant
\[ \widehat{Ric_0}(\omega_J) =\lambda \omega_I, \qquad \widehat{Ric_0}(\omega_I)= \lambda \omega_J, \]
and
\[ \widehat{Ric_0}\left(\Lambda^{(2,0)+(0,2)}_J M\right)= \widehat{Ric_0}\left(\Lambda^{(2,0)+(0,2)}_I M\right)=0. \]

To obtain the first claim, we decompose $\alpha_{\pm}$ according to \eqref{types} 
\[\alpha_+ = \frac{t_1}{2} \omega_J + \alpha^{(2,0)+(0, 2)}_J, \qquad \alpha_- = \frac{t_2}{2} \omega_I + \alpha^{(2,0)+(0, 2)}_I\]
and notice that  in the ambi-K\"ahler situation described above, we have
\[ 
\begin{split}
R(\alpha_+)  &=  \frac{\Scal_gt_1}{8}\omega_J +  \frac{\lambda t_1}{2} \omega_I, \\
  R(\alpha_-)&= \left(\frac{\Scal_gt_2}{24}+ \frac{\kappa_It_2}{12}\right)\omega_I + \left(\frac{\Scal_g}{12} - \frac{\kappa_I}{12}\right)\alpha^{(2,0)+ (0,2)}_I + \frac{\lambda t_2}{2} \omega_J, \end{split}\]
and
\[||\omega_J||^2=||\omega_I||^2=2, \qquad ||\alpha^{(2,0)+(0, 2)}_I ||^2_g =  \frac{(1 - t_2^2)}{2}. \]
We then have
\[
 \begin{split}
K(\alpha) &=   \langle R(\alpha_+), \alpha_+\rangle_g  + \langle R(\alpha_-), \alpha_-\rangle_g  + 2\langle R(\alpha_+), \alpha_-\rangle_g \\
 &=\frac{Scal_g}{8}t_1^2 + \lambda t_1t_2 + \left(\frac{Scal_g}{24}+ \frac{\kappa_I}{12}\right)t_2^2 + \left(\frac{\Scal_g}{12}- \frac{\kappa_I}{12}\right)||\alpha^{(2,0)+(0, 2)}_I||^2_g\\
 &=\frac{Scal_g}{8}t_1^2 + \lambda t_1t_2  + \frac{\kappa_I}{8} t_2^2 + \left(\frac{\Scal_g}{24} - \frac{\kappa_I}{24}\right).
 \end{split} \]
%\[
% \begin{split}
%K(\alpha) &=   \langle R(\alpha_+), \alpha_+\rangle_g  + \langle R(\alpha_-), \alpha_-\rangle_g  + 2\langle R(\alpha_+), \alpha_-\rangle_g \\
% &=\frac{Scal_g}{8}t_1^2 + \lambda t_1t_2 + \left(\frac{Scal_g}{24}+ \frac{\kappa_I}{12}\right)t_2^2 \\
% 	& \hspace{2in} + \left(\frac{\Scal_g}{12}- \frac{\kappa_I}{12}\right)||\alpha^{(2,0)+(0, 2)}_I||^2_g\\
% &=\frac{Scal_g}{8}t_1^2 + \lambda t_1t_2  + \frac{\kappa_I}{8} t_2^2 + \left(\frac{\Scal_g}{24} - \frac{\kappa_I}{24}\right).
% \end{split} \]
For the second claim, if $X$ is a unit vector and $\alpha = (X\wedge JX)^{\flat}$ is the corresponding $2$-form, we notice that $\alpha_+=\frac{1}{2}\omega_J$ and thus $R(X, JX, X, JX)= f_s(1, t)$.                                   
\end{proof}

\subsection{Positivity of the curvature of the Cao type GKRS} The metrics given by the orthotoric ansatz \eqref{Taub-NUT-extension} (defined for arbitrary smooth functions $F_1(\xi_1)<0 < F(\xi_2)$ and variables $\xi_1 < \xi_2$) are examples of ambi-K\"ahler metrics with ambihermitian Ricci tensor, see \cite[p.~132]{ACG2}.  As a mater of fact, the formulae there (with $q=1$) show that the negative K\"ahler metric  is 
\[\tilde g= \frac{1}{(\xi_1-\xi_2)^2} g,  \qquad \tilde \omega_I = \frac{d\xi_1\wedge(dt_1 + \xi_2 dt_2)}{(\xi_1-\xi_2)^2}-\frac{d\xi_2\wedge(dt_1 + \xi_1 dt_2)}{(\xi_1-\xi_2)^2},  \] the  scalar curvature of $g$ is (see \cite[Lemma 9]{ACG0} or \cite[(79)]{ACG})
\[Scal_g = -\left(\frac{F_1''(\xi_1) - F_2''(\xi_2)}{\xi_1 - \xi_2}\right),\]
whereas the scalar curvature of $\tilde g$ is 
\[ 
\begin{split}
Scal_{\tilde g} =  & -(\xi_1-\xi_2)(F_1''(\xi_1) - F_2''(\xi_2)) \\
                           &+ 6(F_1'(\xi_1) + F_2'(\xi_2)) -12\left(\frac{F_1(\xi_1) - F_2(\xi_2)}{(\xi_1 - \xi_2)}\right).
\end{split} \]
We thus get for the conformal scalar curvature of $(g, I)$:  
\[ \begin{split}
\kappa_I = &  \frac{1}{(\xi_1-\xi_2)^2} Scal_{\tilde g} \\
= & -\left(\frac{F_1''(\xi_1) - F_2''(\xi_2)}{(\xi_1-\xi_2)}\right) +  6\left(\frac{F_1'(\xi_1)+ F_2'(\xi_2)}{(\xi_1-\xi_2)^2}\right) 
  -12\left(\frac{F_1(\xi_1) - F_2(\xi_2)}{(\xi_1-\xi_2)^3}\right).
\end{split} \]
%\[ \begin{split}
%\kappa_I = &  \frac{1}{(\xi_1-\xi_2)^2} Scal_{\tilde g} \\
%= & -\left(\frac{F_1''(\xi_1) - F_2''(\xi_2)}{(\xi_1-\xi_2)}\right) +  6\left(\frac{F_1'(\xi_1)+ F_2'(\xi_2)}{(\xi_1-\xi_2)^2}\right) \\
%  &\hspace{1.55in} -12\left(\frac{F_1(\xi_1) - F_2(\xi_2)}{(\xi_1-\xi_2)^3}\right).
%\end{split} \]
From the expression for the Ricci form in \cite[p.~132]{ACG2}, we get for $\lambda$
\[\begin{split}
\pm \lambda &= -\frac{1}{4}\left(\frac{F_1''(\xi_1) + F_2''(\xi_2)}{(\xi_1-\xi_2)}\right)  + \frac{1}{2}\left(\frac{F_1'(\xi_1)-F'_2(\xi_2)}{(\xi_1-\xi_2)^2}\right).
\end{split} \]
Let $f(t)$ be the function of Lemma \ref{sec}.  Up to changing $t \mapsto -t$, we can assume that the expression above is equal to $+\lambda$ for the purposes of studying $f$.~\footnote{One can check directly using the formulas derived below that in fact this choice of $\lambda$ is positive, so that $\lambda = + \frac{1}{4}\left\|Ric_0\right\|_g$ everywhere.}

We now specialize to $F_1(t)=F_2(t)=F(t)=(1-e^{-2a})t - 1 + e^{-2at}$, as in the example \eqref{Cao-extension}. We notice that $F''(t)=4a^2 e^{-2at}>0$ is decreasing and $F'(t)=(1-e^{-2a}) -2ae^{-2at}$ is increasing.  Set 
\[ A:= e^{-2a\xi_1} + e^{-2a\xi_2} > 0, \qquad B:= e^{-2a\xi_1} - e^{-2a\xi_2} > 0. \]
Then we get expressions 
\[ \begin{split}
\Scal_g &= -4a^2\left(\frac{B}{\xi_1-\xi_2}\right) \\
\kappa_I &= -4a^2\left(\frac{B}{\xi_1-\xi_2}\right) - \frac{12}{(\xi_1 -\xi_2)^2}\left(a A +  \frac{B}{\xi_1 - \xi_2}  \right)  \\
\lambda&= -a^2 \left( \frac{A}{\xi_1 - \xi_2} \right) - a \left( \frac{B}{(\xi_1 - \xi_2)^2} \right)
\end{split} \]

\begin{prop}\label{Ric>0}
The Ricci curvature of $g$ is positive. 
\end{prop}
\begin{proof}
We have 
\[ Ric \geq  \left( \frac{Scal_g}{4}  - \lambda \right)g, \]
as $\lambda = \frac{1}{4}|Ric_0|$ is the (unique) positive eigenvalue of the trace-free Ricci tensor. Set $x = \xi_2 - \xi_1$. Then we compute 
\[\begin{split} 
	e^{2a\xi_2} \left( \frac{Scal_g}{4} - \lambda \right) = \frac{a}{x^2}(e^{2ax} - 1) - \frac{2a^2}{x}, 
\end{split}\]
which one readily verifies is strictly positive for all $x \geq 0$. 
\end{proof}

\begin{prop}\label{scal-max}
The scalar curvature of $g$ satisfies 
\[0 < \Scal_g \leq 4a^2\left(1- e^{-2a}\right),\]
and 
\[ \Scal_g = O(\xi_2^{-1}) \qquad \textnormal{as } \xi_2 \to \infty. \]
In particular, $\Scal_g$ attains its maximum at the origin of $\C^2$.
\end{prop}
\begin{proof} The inequality $\Scal_g>0$ follows from Proposition~\ref{Ric>0}.   The other inequality is immediate from the expression for $\Scal_g$ derived above.  As $\xi_2$  is a proper function (this is established in the proof of Lemma~\ref{l:completion})  and $\Scal_g \to 0$ when $\xi_2\to \infty$, we conclude that $\Scal_g$ achieves its maximum on $\C^2$. By \cite[Prop.~4.2]{Cao-Hamilton}  and \cite[Lemma 1]{bryant} the only critical point of $\Scal_g$ coincides with the unique point where the strictly convex soliton function $f=a(\xi_1 + \xi_2)$  achieves its minimum. The latter corresponds to $\xi_1=0$ and $\xi_2=1$, i.e. is the fixed point of the $\T^2$ action. The claim follows.
\end{proof}

\begin{prop}\label{hol-sec}
In the situation described above, we have that $f(t) > 0$. In particular,  in this case the holomorphic sectional curvature of $g$ is positive. 
\end{prop}
The proof of Proposition \ref{hol-sec} is broken up into the following two lemmas 
\begin{lemma}\label{hol-sec-lemma1}
	The function $f(t)$ is positive at $t = \pm 1$.
\end{lemma}
\begin{proof}
For $t = -1$, we see that
%\begin{equation*}
%\begin{split}
%	f(-1) &=  \frac{\Scal_g}{6} + \frac{\kappa_I}{12} - \lambda  \\
%			&= \frac{\Scal_g}{4}   -  \frac{1}{(\xi_1 -\xi_2)^2}\left(a A +  \frac{B}{\xi_1 - \xi_2}  \right) -  a^2 \left( \frac{A}{\xi_1 - \xi_2} \right) + a \left( \frac{B}{(\xi_1 - \xi_2)^2} \right) \\
%			&= -a^2\left( \frac{B}{\xi_1 -\xi_2} \right)   -  \frac{1}{(\xi_1 -\xi_2)^2}\left(a A +  \frac{B}{\xi_1 - \xi_2}  \right) -  a^2 \left( \frac{A}{\xi_1 - \xi_2} \right) - a \left( \frac{B}{(\xi_1 - \xi_2)^2} \right),
%\end{split}
%\end{equation*}
\begin{equation*}
\begin{split}
	f(-1) &=  \frac{\Scal_g}{6} + \frac{\kappa_I}{12} - \lambda  \\
			&= \frac{\Scal_g}{4}   - \frac{aA}{(\xi_1 -\xi_2)^2} -  \frac{B}{(\xi_1 - \xi_2)^3}   -  \frac{a^2 A}{\xi_1 - \xi_2}  +  \frac{aB}{(\xi_1 - \xi_2)^2}  \\
				&= -\frac{a^2B}{\xi_1 -\xi_2}  - \frac{aA}{(\xi_1 -\xi_2)^2} -  \frac{B}{(\xi_1 - \xi_2)^3}   -  \frac{a^2 A}{\xi_1 - \xi_2}  +  \frac{aB}{(\xi_1 - \xi_2)^2} 
			%&= -a^2\left( \frac{B}{\xi_1 -\xi_2} \right)   -  \frac{1}{(\xi_1 -\xi_2)^2}\left(a A +  \frac{B}{\xi_1 - \xi_2}  \right) -  a^2 \left( \frac{A}{\xi_1 - \xi_2} \right) - a \left( \frac{B}{(\xi_1 - \xi_2)^2} \right),
\end{split}
\end{equation*}
so that, again writing $x = \xi_2 - \xi_1$,  we have%~\footnote{VA: how do you see it? I would try to show the inequality below directly. CC: good catch, in fact it's the opposite sign.  It's the slope of the secant line to $e^{-2ax}$,  at $x = \xi_1$ which is of course \emph{increasing} as we increase $\xi_1$. I replaced this with a simpler argument in the same vein as the others.} 
\begin{equation*}
\begin{split}
e^{2a\xi_2}f(-1) &= \frac{a^2(A+B)}{x} + \frac{B}{x^3} - \frac{a(A+B)}{x^2} \\
				&= \frac{2a^2}{x} + \frac{e^{2ax} - 1}{x^3} - \frac{2a}{x^2} \\
				&= \frac{2a^2}{x} + \sum_{k = 2}^\infty \frac{(2a)^k x^{k-3}}{k!} > 0
\end{split}
\end{equation*}
At $t = +1$,  we have that 
\[ f(1) = f(-1) + 2\lambda > 0 \]
since $f(-1) > 0$ and $\lambda \geq 0$.\\
 \end{proof}

  \begin{lemma}
For all $(\xi_1, \xi_2)$ with $x \geq 0$ there are only three possibilities: either $f$ is constant, $f(t)$ has no critical points in $[-1, 1]$, or $f(t^*) > 0$ at a critical point $t^* \in [-1, 1]$. 
\end{lemma}
\begin{proof}
If $(\xi_1, \xi_2)$ are such that $\kappa_I = 0$, then $f''(t) = \frac{\kappa_I}{4} = 0$. 
Otherwise, the function $f(t)$ has a critical point at $t^*$ if and only if 
  \[ t^*   = -4\frac{\lambda }{\kappa_I}. \] 
  Now, if $|t^*| > 1$, then $f$ has no critical points for $t \in [-1,1]$. 
  %and so $f(t)>0$ on the whole interval $[-1,1]$. 
  Otherwise $|t^*| \leq 1$, and evaluating back in $f(t)$ gives 
  \begin{equation*}
\begin{split}
f(t^*) &= \frac{\Scal_g}{6} - \frac{\kappa_I }{24}+ 2\frac{\lambda^2}{\kappa_I} - 4 \frac{\lambda^2}{\kappa_I} \\
	&= \frac{\Scal_g}{6} - \frac{\kappa_I }{24} -  2\frac{\lambda^2}{\kappa_I}. \\
	&= \frac{\Scal_g}{6} - \frac{\kappa_I }{24} + \frac{\lambda}{2}  t^*.
\end{split}
\end{equation*}
Therefore
\[ f(t^*) \geq   \frac{\Scal_g}{6} - \frac{\kappa_I }{24}-  \frac{|\lambda|}{2} .   \] 
Therefore we can compute 
\[ \begin{split}
  2f(t^*)  &\geq   \frac{\Scal_g }{3}- \frac{\kappa_I}{12}  -\lambda   \\ 
 			&=    \frac{a^2(B-A)}{\xi_2-\xi_1}  +  \frac{a(A + B)}{(\xi_2 - \xi_1)^2}  -  \frac{B}{(\xi_2 - \xi_1)^3} 
\end{split} \]
Re-expressing the above in terms of the variable $x$, we get
\[ \begin{split}
2e^{2a\xi_2} f(t^*) &\geq    \frac{2ae^{2ax}}{x^2}   - \frac{2a^2}{x}-   \frac{e^{2ax}- 1}{x^3} \\
				&= \sum_{k = 0}^\infty \left( 1 - \frac{1}{k+3}\right)\frac{(2a)^{k+3}x^{k}}{k!} > 0.
\end{split} \]
\end{proof}
Using this, we can in fact show that $g$ has positive sectional curvature. Recall that the minimum of the sectional curvature is controlled by the minimum of
\[f_{s}(t_1, t_2):= \left(\frac{\Scal_g }{8} \right)t_1^2+ \lambda t_1t_2  + \left(\frac{\kappa_I }{8}\right)t_2^2+ \left( \frac{\Scal_g - \kappa_I}{24} \right), \]
over $(t_1, t_2)\in [-1,1]\times [-1,1].$
\begin{prop}\label{full-sec}
The function $f_s(t_1, t_2) > 0$,  so that the full sectional curvature of $g$ is positive.  
\end{prop}
The proof of Proposition \ref{full-sec} will be a combination of the following lemmas. 
\begin{lemma}\label{conformalnonneg}
The conformal scalar curvature $\kappa_I$ is nonnegative
\end{lemma}
\begin{proof}
We evaluate 
\[\begin{split}  e^{2a\xi_2}\kappa_I &= 4a^2 \frac{e^{2ax} - 1}{x} + 12 \frac{e^{2ax} - 1}{x^3}  - 12a \frac{e^{2ax} + 1}{x^2} \\
			&= 4a^2 \sum_{k = 1}^\infty \frac{(2a)^kx^{k-1}}{k!} + 12 \sum_{k = 3}^\infty \frac{(2a)^k x^{k-3}}{k!} - 12a\sum_{k=2}^\infty \frac{(2a)^k x^{k-2}}{k!} \\
			&= \sum_{k = 0}^\infty \left(1 + \frac{12}{(k+3)(k+2)} - \frac{6}{k+2} \right) \frac{(2a)^{k+2} x^k}{(k+1)!} \\
			&=\sum_{k = 0}^\infty \left(\frac{k(k-1)}{(k+3)(k+2)} \right) \frac{(2a)^{k+2} x^k}{(k+1)!},
\end{split}\]
from which we see that $e^{2a\xi_2}\kappa_I \geq 0$ and vanishes precisely at $x = 0$. 
\end{proof}
\begin{lemma}\label{fullsec-gradient} The $1$-form  $d f_s $ vanishes only at $(t_1, t_2) = (0,0)$, unless 
\[ \kappa_I Scal_g = 16 \lambda^2, \]
in which case it vanishes along the line 
\[  t_1\Scal_g + 4 t_2 \lambda =  t_2\kappa_I+ 4 t_1 \lambda = 0.\]
Moreover $f_s > 0$ at any critical point. 
\end{lemma}
\begin{proof}
We compute that 
\[ \frac{\partial f_s}{\partial t_1} = \left(\frac{\Scal_g}{4}\right)t_1 +  \lambda t_2 , \qquad  \frac{\partial f_s}{\partial t_2} =  \left(\frac{\kappa_I}{4}\right)t_2+ \lambda t_1 ,\]
which yields the first few claims.  For the last part, we claim that the value of $f_s$ at the critical point $(0, 0)$ satisfies 
  \[f_s(0,0) =  \frac{Scal_g - \kappa_I}{24} \geq \frac{a^3e^{-2a}}{3} > 0. \]
 To see this, we can compute 
\[\begin{split}  e^{2a\xi_2} \left( \frac{Scal_g - \kappa_I}{24} \right) &=  a\frac{e^{2ax} + 1}{2x^2} - \frac{e^{2ax} - 1}{2x^3}  \\
				&= \sum_{k=3}^\infty \frac{a^k 2^{k-2}}{(k-1)!} \left(1 - \frac{2}{k} \right) x^{k - 3},
 \end{split} \]
 from which the result is immediate.  In case that $ t_1\Scal_g + 4 t_2 \lambda =  t_2\kappa_I+ 4 t_1 \lambda = 0$ for all $t_1, t_2$, we observe that 
 \[ f_s(t_1, t_2) = \frac{t_1}{8}( t_1\Scal_g + 4 t_2 \lambda)  + \frac{t_2}{8} (t_2\kappa_I+ 4 t_1 \lambda) + f_s(0,0)  = f_s(0,0) > 0.    \]
\end{proof}
\begin{lemma}\label{ahs-min}
If  $\bar{f}(t) = \left(\frac{Scal_g}{8}\right) t^2 + \lambda t  + \frac{Scal_g + 2\kappa_I}{24} > 0 $
for all $t \in [-1,1]$, then $f_s(t_1, t_2) > 0$ for all $(t_1, t_2) \in [-1,1]\times [-1,1]$. 
\end{lemma}
\begin{proof}
By Lemma \ref{fullsec-gradient},  $f_s$ is positive if and only if it is positive on the boundary.  Moreover, the function $f_s(t_1, t_2)$ is positive on the boundary component $\{\pm 1 \} \times [-1, 1]$. Indeed, this follows from the bound on the holomorphic sectional curvature, as we have $f_s(1,t) = f(t) > 0$ and $f(-1, t) = f(-t) > 0$.  In a similar fashion we see that $\bar{f}(\pm t) = f_s(t, \pm 1)$. 
\end{proof}
\begin{lemma}
The function $\bar{f}(t)$ is strictly positive for all $a > 0$ and $t \in [-1, 1]$. 
\end{lemma}
\begin{proof}
First, we notice that $\bar{f}''(t) =  \frac{Scal_g}{4} > 0$, so that $\bar{f}$ is convex on $[-1, 1]$. Therefore $\bar{f}$ is minimized at a critical point $t^*$.  We compute directly that~
\begin{equation*}
t^* = -\frac{4\lambda}{Scal_g} = \frac{1}{ax} - \frac{e^{2ax} + 1}{e^{2ax} - 1},
\end{equation*}
which satisfies $ -1 < t^* \leq 0$ for all $x \geq 0$.  Hence such a critical point always exists, and so $\bar{f}$ is minimized at 
\[ 	\bar{f}(t^*) = \frac{t}{8}(Scal_g t) +  \lambda t^* + \frac{Scal_g + 2\kappa_I}{24}  =  \frac{\lambda}{2} t^* +  \frac{Scal_g + 2\kappa_I}{24}.   \]
Next we calculate that 
\begin{equation} \label{expression1}
e^{2a\xi_2} \left( \frac{Scal_g + 2\kappa_I - 12 \lambda}{12}\right) =   - \frac{2a^2}{x} - \frac{3a}{x^2} - \frac{a}{x^2}e^{2ax}  + \frac{2}{x^3}(e^{2ax} - 1).  
\end{equation} 
 Using \eqref{expression1}, one can compute in a similar way to the previous calculations that
 \[ \begin{split}2e^{2a\xi_2} \bar{f}(t^*) &= \frac{1}{x^3}(e^{2ax} - 1)  -  \frac{4a^2	}{x}\left(\frac{e^{2ax}}{e^{2ax} -1}\right).
 \end{split} \]
We then compute that $h(x) := 2(e^{2ax} - 1) e^{2a\xi_2} \bar{f}(t^*)$ can be expressed
 \[ \begin{split} h(x) &= \frac{1}{x^3}(e^{2ax} - 1)^2  -  \frac{4a^2 }{x} e^{2ax}\\
 			&= \frac{1}{x^3}\left(e^{4ax} - 2e^{2ax} + 1 \right) - \frac{4a^2 }{x} e^{2ax} \\
 			&= \sum_{k = 0}^\infty \frac{2^{2k}a^k x^{k-3}}{k!} - \sum_{k = 0}^\infty \frac{2^{k+1}a^k x^{k-3}}{k!} + \frac{1}{x^3} - \sum_{k = 0}^\infty \frac{2^{k+2}a^{k+2} x^{k-1}}{k!} \\
 			&= \sum_{k = 0}^\infty \left( \frac{2^{2(k+3)} a^{k+3}}{(k+3)!} -  \frac{2^{k+4} a^{k+3}}{(k+3)!} -  \frac{2^{k+3} a^{k+3}}{(k+1)!}\right) x^k \\
 			&=  \sum_{k = 0}^\infty \frac{2^{k+3} a^{k+3}}{(k+3)!} \left( 2^{k+3} - 2 - (k+3)(k+2)\right) x^k,
 \end{split} \]
which is strictly positive for $x > 0$.  We have already seen that when $x = 0$ we have that $t^* = 0$, which implies that $\bar{f}(t^*) = \frac{Scal_g + 2\kappa_I}{24}$ is strictly positive.  Of course, this can also be seen directly from the above. Indeed, the constant term above vanishes, whereas the linear term is given by $2\frac{(2a)^4}{4!} > 0$.  Since $e^{2ax} - 1$ vanishes to order 1 at $x = 0$,  it follows that $f(t^*)$ is strictly positive also at $x = 0$. 
\end{proof}

\section*{Acknowledgements} {We are grateful to the anonymous referees for their careful reading of the manuscript and suggesting a number of valuable improvements, to R.~Conlon and S. Sun for their interest, insightful remarks and for bringing the reference \cite{Li} to our attention, and to  L. Foscolo, G. Oliviera,  R. Sena Dias,   J. Streets and Y. Ustinovskiy for discussions and useful remarks.}

\end{document}